\documentclass{article}
\usepackage{arxiv}
\usepackage[utf8]{inputenc} 
\usepackage[T1]{fontenc}    
\usepackage{hyperref}       
\usepackage{url}            
\usepackage{booktabs}       
\usepackage{amsfonts}       
\usepackage{nicefrac}       
\usepackage{microtype}      
\usepackage[english]{babel}
\hypersetup{colorlinks=false}
\usepackage{lscape}
\usepackage{amsmath}
\usepackage{enumerate}
\usepackage{amsthm}
\usepackage{amssymb}
\usepackage[final]{graphicx}
\usepackage{mathtools}
\usepackage{subcaption}

\graphicspath{{images/}}

\title{Feasible Inference for Stochastic Volatility in Brownian Semistationary Processes}

\DeclareMathOperator{\Var}{Var}
\DeclareMathOperator*{\plim}{plim}

\newtheorem{theorem}{Theorem}[section]
\newtheorem{proposition}{Proposition}[section]

\theoremstyle{definition}

\theoremstyle{remark}

\author{
  {Phillip Murray}\thanks{Corresponding author: Department of Mathematics, Imperial College London, South Kensington Campus, London SW7 2AZ, UK}\\
  Imperial College London\\
  \texttt{phillip.murray18@imperial.ac.uk}\\
  \And
    Riccardo Passeggeri\\
  Sorbonne University\\
  \texttt{riccardo.passeggeri14@imperial.ac.uk}\\
  \And
    Almut E.D. Veraart \\
  Imperial College London\\
  \texttt{a.veraart@imperial.ac.uk}\\
  \And
    Mikko S. Pakkanen \\
  Imperial College London\\
  \texttt{m.pakkanen@imperial.ac.uk}\\  
}

\begin{document}
\maketitle

\begin{abstract}
This article studies the finite sample behaviour of a number of estimators for the integrated power volatility process of a Brownian semistationary process in the non-semi-martingale setting. We establish three consistent feasible estimators for the integrated volatility, two derived from parametric methods and one non-parametrically. We then use a simulation study to compare the convergence properties of the estimators to one another, and to a benchmark of an infeasible estimator. We further establish bounds for the asymptotic variance of the infeasible estimator and assess whether a central limit theorem which holds for the infeasible estimator can be translated into a feasible limit theorem for the non-parametric estimator. 
\end{abstract}

\section{Introduction}
In this article we study the asymptotic behaviour of the realised power variation in relation to a class of stochastic processes known as \textit{Brownian semistationary} ($\mathcal{BSS}$) processes. These processes were first introduced by Barndorff-Nielsen and Schmiegel \cite{barndorff2007ambit} \cite{barndorff2009brownian} and we will focus on a particular subclass of these processes, namely zero mean and driftless variants expressed by an integral representation 

$$ Y_t := \int_{-\infty}^t g(t-s) \sigma_s \ \mathrm{d}W_s, $$
where $W$ is a two-sided Brownian motion which provides the driving noise, $\sigma$ is a stochastic volatility process, and $g$ is a deterministic kernel function that specifies the historical dependence of the process, both in terms of the short term smoothness of the paths, and the long term memory. $\mathcal{BSS}$ processes can be seen as volatility modulated continuous time moving average processes, and fit into a much broader class of L\'{e}vy driven processes, see \cite{barndorff2018ambit} for a full exploration. They are of particular use in the modelling of turbulence of physical systems  \cite{barndorff2007ambit} \cite{corcuera2013asymptotic},  and also have applications in finance, such as models of energy prices \cite{barndorff2013modelling} \cite{bennedsen2017rough}.

An important quantity of interest for these processes is the \textit{integrated power volatility process},

$$ V(Y,p)(t) := m_p \int_0^t |\sigma_s|^p \ \mathrm{d}s,$$
where $m_p$ is the $p$th absolute moment of a standard normal distribution. Of particular interest is the case $p=2$ which gives the integrated squared volatility process. When $Y$ is a semi-martingale, it has been well established that the \textit{realised power variation}, defined by

$$ V(Y, p; \Delta_n)(t) := \sum_{i = 1}^{ \lfloor t/\Delta_n \rfloor} |\Delta_i^nY |^p, \qquad \Delta_i^nY = Y_{i\Delta_n} - Y_{(i-1)\Delta_n}, $$
produces a consistent estimator for the integrated volatility process \cite{jacod2008asymptotic}. In the non-semi-martingale setting, the realised power variation alone is no longer a consistent estimator for the integrated volatility. We instead scale the realised power variation using a scale factor $\tau_n$, that depends on the short term structure of the process, in particular, on the smoothness properties of the kernel function. 

It has been proven that in the non-semi-martingale case, this suitably scaled realised power variation is a consistent estimator for the integrated power volatility process, and also obeys a central limit theorem \cite{corcuera2013asymptotic}. This estimator is, however, \textit{infeasible} in the sense that the scale factor must be known, which in turn requires knowing the kernel functional form and parameters, which in general, for real-world applications where we have observed data, it will not be. We therefore refer to these previous results as the \textit{infeasible weak law} and the \textit{infeasible central limit theorem} respectively. This article extends these results to the feasible case, where the scale factor is also estimated from the data, enabling us to establish a \textit{feasible weak law}. We briefly describe two parametric methods which produce feasible consistent estimators after specification of a kernel function, and then focus on  establishing a non-parametric feasible estimator, which requires no assumption on the functional form of the kernel. The question of whether the limit theorem can be translated into a \textit{feasible central limit theorem} is also explored, and we assess how estimation of the scale factor effects the nature of the convergence, both in terms of the rate, and also the limiting distribution. 

The theoretical results are then assessed through a simulation study, to assess the performance of each of the asymptotic results in finite sample behaviour. The convergence relies on both \textit{infill} and \textit{long-span} asymptotics. We simulate a large number of sample paths from a $\mathcal{BSS}$ process, and use the paths to estimate the integrated power volatility process, using both the infeasible and feasible estimators and compare the result, in order to establish at what frequencies the weak law and limit theorem begin to converge, using the infeasible estimator as the benchmark of the `ideal' scenario, and comparing feasible estimator to this. We further test experimentally whether the limit theorem holds for our non-parametric estimator. The code implementation has been made into a new \texttt{R} package, \texttt{BSS}, available on CRAN \cite{package}.

This article is structured as follows. In section 2, we begin by outlining the background theory into $\mathcal{BSS}$ processes and establish the core asymptotic theory, first in the infeasible setting, and then extend it to the feasible case. We prove the consistency of these estimators and hence establish their use in our feasible estimation of the integrated volatility process. In section 3 we provide details of the simulation study methods, and discuss the results of the simulation study in section 4. We then provide an application of the methods to a dataset in section 5, before concluding. Proofs of the key result of the feasible weak law and central limit theorem are provided at the end of the article.

\section{Model and theoretical results}
\subsection{Brownian semistationary processes}
Let $(\Omega, \mathcal{F}, \{\mathcal{F}_t\}_{t \in \mathbb{R}}, \mathbb{P})$ be a filtered probability space satisfying the conditions of completeness and right continuity of the filtration. A Brownian semistationary $(\mathcal{BSS})$ process is a stochastic process $Y =\{Y\}_{t\in \mathbb{R}_{+}}$ defined on this space by 

$$ Y_t = \mu + \int_{-\infty}^t h(t-s)a_s \ \mathrm{d}s + \int_{-\infty}^t g(t-s)\sigma_s \  \mathrm{d}W_s$$
where $\mu \in \mathbb{R}$ is a constant, $h \in L^1(\mathbb{R})$ and $g \in L^2(\mathbb{R})$ are non-negative deterministic kernel functions satisfying $h(t) = g(t) = 0$ for $t \leq 0$. The stochastic processes $a = \{a_t\}_{t \in \mathbb{R}}$ and $\sigma = \{\sigma_t\}_{t \in \mathbb{R}}$ are adapted to the filtration $\{\mathcal{F}_t\}_{t \in \mathbb{R}}$ such that all integrals exist and are well defined. Furthermore, $W = \{W_t\}_{t \in \mathbb{R}}$ is a two-sided Brownian motion which is also adapted to the filtration $\{\mathcal{F}_t\}_{t \in \mathbb{R}}$, where by two-sided we mean that we take two independent copies of a standard Brownian motion $W^1_t$ and $W^2_t$ each defined on $\mathbb{R}_{+}$ as usual, and define $W_t = W^1_t$ for $t \geq 0$ and $ W_t = W^2_{-t_{-}}$ for $t < 0$. 

The stochastic processes $a = \{a_t\}_{t \in \mathbb{R}}$ and $\sigma = \{\sigma_t\}_{t \in \mathbb{R}}$ are referred to as the \textit{drift} and \textit{volatility} processes respectively.  Clearly, the constant $\mu$ does not affect the stochastic behaviour of the process $Y$ and we are primarily focussed on inference on the underlying volatility process, we will consider only driftless $\mathcal{BSS}$ processes, which for convenience have zero mean, and hence we will set $\mu \equiv 0$ and $a \equiv 0$ for all $t \in \mathbb{R}_{+}$. Therefore, we will consider the simplified class of $\mathcal{BSS}$ processes defined by 

$$ Y_t =\int_{-\infty}^t g(t-s)\sigma_s \ \mathrm{d}W_s$$
where $\sigma$ is a positive, c\`{a}dl\`{a}g process, and to ensure that this integral is well-defined, we will require that the square integrable kernel function $g$ is such that 

$$ \int_{-\infty}^t g(t - s)^2 \sigma_s^2 \ \mathrm{d}s < \infty \qquad \text{a.s.}$$
for all $t \geq 0$. Furthermore, we will assume that the volatility process $\sigma$ is stationary with finite second moments. 

Associated with the $\mathcal{BSS}$ process $Y$ is a stationary centred Gaussian process $X = \{X_t\}_{t \in \mathbb{R}}$, called the Gaussian core of $Y$, which is given by

$$X_t = \int_{-\infty}^t g(t-s) \ \mathrm{d}W_s.$$

It is useful to define the autocovariance function for the Gaussian core, for lag $h \geq 0$

$$ \gamma_X(h) = \int_0^\infty g(x)g(x + h) \ \mathrm{d}x,$$
and the autocorrelation kernel of the $\mathcal{BSS}$ process, given by

$$\rho_X(h) = \frac{\int_0^\infty g(x)g(x + h) \  \mathrm{d}x}{\int_0^\infty g(x)^2 \ \mathrm{d}x}.$$

It is straightforward to show that $\gamma_Y(h) = \mathbb{E}[\sigma_0^2] \gamma_X(h)$. This result in particular demonstrates the role of the kernel function $g$ in the `memory' or correlation structure of the $\mathcal{BSS}$ process, since we have $\rho_Y(h) = \rho_X(h)$ for all $h \geq 0$, and thus the correlation structure of $Y$ is identical to the correlation kernel of the underlying Gaussian core $X$, and does not depend on the stochastic volatility process $\sigma$. A further quantity which will prove to be important is the \textit{second-order structure function}, or \textit{variogram}, 

$$R(t) = \mathbb{E}[(X_t - X_0)^2], \qquad t \geq 0,$$
of the Gaussian core of the process.

\subsection{Kernel functions}
We require various conditions on the kernel 
 $g:\mathbb{R} \rightarrow \mathbb{R}$ in order for the $\mathcal{BSS}$ process to be well behaved \cite{barndorff2014assessing}. It holds that

\begin{enumerate}[(i)]
\item $g(x) = x^\alpha L_g(x)$, where $L_g$ is some function that is slowly varying at $0$, that is, the ratio $L_g(tx) / L_g(x) \rightarrow 1$ as $x \rightarrow 0$ for any number $t > 0$. 

\item The function $g$ is continuously differentiable on $(0, \infty)$ with $g'(x) = x^{\alpha -1}L_{g'}(x)$ and we have $g' \in L^2(\epsilon, \infty)$ for any $\epsilon > 0$. Moreover, $g'$ is non-decreasing on $(a, \infty)$ for some $a>0$, so that the derivative is ultimately monotonic. 

\item $\int_1^\infty g'(s)^2 \sigma_{t-s}^2 \ \mathrm{d}s < \infty$ a.s. for any $t>0$.

Furthermore, the second-order structure function $R$ satisfies:

\item $R(t) = t^{2\alpha + 1} L_R(t).$
 
\item $R''(t) = t^{2\alpha -1} L_{R''}(t).$

\item For some $b \in (0,1),$ 

$$\limsup\limits_{s\downarrow 0} \sup\limits_{t \in [s,s^b]} \left\vert \frac{L_{R''(t)}}{L_R(s)} \right\vert < \infty.$$

\end{enumerate}

The parameter $\alpha$ is referred to as the \textit{smoothness parameter}, as it controls the local behaviour of the paths. A consequence of Knight's theorem \cite{knight1992foundations} is that for $\alpha \in (-\frac{1}{2},0) \cup (0, \frac{1}{2})$, the process $Y$ will not be a semi-martingale \cite{barndorff2009brownian}, which is of particular interest in many applications such as the study of turbulence, such as $\alpha = -1/6$ corresponding to Kolomogorov's scaling law in turbulence \cite{kolmogorov1941dissipation}. There is one particular kernel choice which is widely used and of interest to us, the gamma kernel defined as 

$$g(x) = x^\alpha e^{-\lambda x}$$
with parameters $\alpha \in (-\frac{1}{2},0) \cup (0, \frac{1}{2})$ and $\lambda > 0$. It is a well studied kernel choice in the $\mathcal{BSS}$ literature, because it relates closely to the study of turbulence, and also provides a generalisation of the Ornstein-Uhlenback process with the inclusion of the roughness parameter which takes it out of the semi-martingale case - the Ornstein-Uhlenbeck being a special case of the gamma kernel with parameter $\alpha = 0$.  The gamma kernel decays exponentially to zero as $x \rightarrow \infty$, with the rate controlled by the parameter $\lambda$. See, for example, \cite{barndorff2012notes} for more detailed exploration of the gamma kernel.

\subsection{Asymptotic theory}

We now move on to establishing the core finite sample behaviour of the $\mathcal{BSS}$ process observed in discrete time. Suppose that for some $T>0$, we have observations on a fine grid $Y_{i\Delta_n}$ for $i = 0, ..., \lfloor T/\Delta_n \rfloor$ with $\Delta_n \rightarrow 0$ as $n \rightarrow \infty$, and define the first order increments to be

$$\Delta_i^nY = Y_{i\Delta_n} - Y_{(i-1)\Delta_n}.$$

Define further the $p$th order \textit{realised power variation} of $Y$ at time $0<t\leq T$, observed at frequency $\Delta_n$, by

$$ V(Y, p; \Delta_n)(t) :=  \sum_{i = 1}^{ \lfloor t/\Delta_n \rfloor} |\Delta_i^nY |^p.$$

Now, in the case that $Y$ is a semi-martingale, $ \Delta_n^{1-p} V(Y,p;\Delta_n)(t)$ is a consistent estimator for 

$$V(Y,p)(t) = g(0+)^2 m_p \int_0^t |\sigma_s|^p \ \mathrm{d}s$$ 
as $\Delta_n \rightarrow 0$, where $m_p := \mathbb{E}[|U|^p]$ for a standard normal random variable $U \sim N(0,1)$ \cite{barndorff2006central} \cite{jacod2008asymptotic}. However, outside the semi-martingale setting these results for convergence of the realised power variation no longer hold. Instead, alternative convergence results have been established, in which we scale the realised power variation by a suitable factor which depends on $\Delta_n$ and the scaling properties of $Y$, which we will now outline.

We first define the quantity $\tau_n$ from the second order structure function of the Gaussian core $X$ as $\tau_n := \sqrt{R(\Delta_n)}$, and it will be convenient to note that $\tau_n = \sqrt{2\gamma_X(0) - 2\gamma_X(\Delta_n)}$. Then a number of asymptotic results for the convergence of the realised power variation have been proven in the non-semi-martingale case.

\subsection{Infeasible weak law}
The first asymptotic result that has been established is a form of a weak law of large numbers for the suitably scaled realised power variation \cite{corcuera2013asymptotic}.  

\begin{theorem}
	\label{thm:infeasible_weak_law}
For any $p >0$, and fixed $t>0$, 

$$\Delta_n \tau_n^{-p} V(Y,p;\Delta_n)(t) \xrightarrow{u.c.p.} V(Y,p)(t) := m_p \int_0^t |\sigma_s|^p \ \mathrm{d}s,$$
as $\Delta_n \rightarrow 0$, and where $m_p := \mathbb{E}[|U|^p]$ for a standard normal random variable $U \sim N(0,1).$ The convergence is uniform on compacts in probability. 
\end{theorem}

We note that the above theorem shows that $m_p^{-1} \Delta_n \tau_n^{-p} V(Y,p;\Delta_n)(t)$ is a consistent estimator for the stochastic quantity $\int_0^t |\sigma_s|^p \ \mathrm{d}s$. The law is infeasible in the sense that it requires knowledge of the scale factor $\tau_n$, which in turn depends on knowing the kernel function $g$. 

\subsection{Feasible weak law}
In the case where $g$ is unknown, the scaling factor $\tau_n$ must be estimated from the observations $Y_{i\Delta_n}, i = 0, 1, ..., \lfloor T/\Delta_n \rfloor$. Then the following proposition enables us to `plug in' an estimator of $\tau_n$ and establish a feasible weak law for the realised power variation. 

\begin{proposition}
	\label{prop:feasible_weak_law}
Suppose that for any fixed $n$, we may construct an estimator $\hat{\tau}_n^N$ based on $N \in \mathbb{N}$ observations of $Y$ at frequency $\Delta_n$, which is consistent estimator for $\tau_n$. That is, suppose that for fixed $n \in \mathbb{N}$ we have

\begin{equation}
	\label{eq:long_span_consistency}
\hat{\tau}_n^N \xrightarrow{\mathbb{P}} \tau_n \qquad \text{as } N \rightarrow \infty. 	
\end{equation}

Then

$$\plim_{n \rightarrow \infty} \plim_{N \rightarrow \infty} \Delta_n (\hat{\tau}^N_n)^{-p} V(Y,p;\Delta_n)(t) = V(Y,p)(t) = m_p \int_0^t |\sigma_s|^p \ \mathrm{d}s,$$
where $\plim$ denotes a limit in probability. 
\end{proposition}

\begin{proof}
For notational convenience, write $\hat{\tau}_n = \hat{\tau}_n^N$, and then we may write 

$$ \Delta_n (\hat{\tau}_n)^{-p} V(Y,p;\Delta_n)(t) = \Delta_n \tau_n^{-p} V(Y,p;\Delta_n)(t) \frac{\hat{\tau}_n^{-p}}{\tau_n^{-p}}.$$

Then, since $\hat{\tau}_n^{-p} / \tau_n^{-p} \rightarrow 1$ in probability as $N \rightarrow \infty,$ by using Theorem \ref{thm:infeasible_weak_law} and applying the continuous mapping theorem, the product of these terms will also converge to $V(Y,p)(t)$ by first taking the limit as $N \rightarrow \infty$, and then the limit as $n \rightarrow \infty$.
\end{proof}

\subsection{Estimation of the scale factor}
The above feasible weak law tells us that we can replace the true scale factor $\tau_n$ with a consistent estimator, and the same asymptotic properties of the realised power variation will hold. We may construct a non-parametric estimator of the scale factor as follows. Consider an estimator for $\tau_n$, based on $N$ observations, of the form

$$\tau_n^N := \sqrt{\frac{1}{N} \sum_{i=1}^{N} (\Delta_i^n X)^2}.$$

We see that $\tau_n^N$ depends only on observations of the Gaussian core. Then we may show that $\tau_n^N$ converges in probability to $\tau_n$, as $N \rightarrow \infty.$ In fact, we can prove a stronger result, a limit theorem for this non-parametric estimator, and then the convergence in probability follows as a consequence. Let $\hat{\gamma}(\cdot ) := \hat{\gamma}^N(\cdot)$ be the sample autocovariance of the Gaussian core based on $N \in \mathbb{N}$ observations, defined by

$$\hat{\gamma}(h\Delta_n ) = \frac{1}{N} \sum_{i=1}^{N} X_{i\Delta_n} X_{(i + h)\Delta_n}.$$

Now we use a result for a more general class of processes, a L\'{e}vy driven continuous time moving average process of the form

$$X_t = \int_{-\infty}^t g(t-s) \ \mathrm{d}L_s$$
where $L$ is a two-sided L\'{e}vy process, so that, in particular the Gaussian core of a $\mathcal{BSS}$ process is a subset of this class of processes. For the core of the process, the following result has been established \cite{cohen2013central} regarding the limiting distribution of the sample autocovariance of this class of processes. 

\begin{theorem}
Fix $\Delta_n$ and let $X$ be a L\'{e}vy driven continuous time moving average process, and suppose that 

$$ \sum_{k=-\infty}^{\infty} \left( \int_{-\infty}^{\infty} |g(s)g(s + k\Delta_n)| \ \mathrm{d}s \right)^2 < \infty.$$

Then we have for each $h \in \mathbb{N}$
 
$$ \sqrt{N} \left( \hat{\gamma}(0) - \gamma(0), ...,  \hat{\gamma}(h\Delta_n) - \gamma(h\Delta_n) \right) \xrightarrow{d} N(0, V) $$
as $N \rightarrow \infty$, where $V = (V_{pq}) \in \mathbb{R}^{h+1 \times h+1}$ is the covariance matrix defined for $p,q = 0, ..., h$ by 

$$V_{pq} = (\mathbb{E}[L_1^4] - 3(\mathbb{E}[L_1^2])^2) \int_0^{\Delta_n} \bar{g}_p(x) \bar{g}_q(x) \ \mathrm{d}x \ + \sum_{k = -\infty}^\infty [ \gamma(k\Delta_n)\gamma((k-p+q)\Delta_n) + \gamma((k+q)\Delta_n)\gamma((k-p)\Delta_n)]$$
and where the function $\bar{g}_q$ is defined as 

$$\bar{g}_q(x) = \sum_{k = -\infty}^\infty g(x + k\Delta_n) g( x + (k + q)\Delta_n).$$
\end{theorem}

An immediate corollary of this result is that when $X$ is the Gaussian core of a $\mathcal{BSS}$ process then we can disregard the first term in the expression for $V_{pq}$ since the coefficient is identically zero, and thus have

$$V_{pq} = \sum_{k = -\infty}^\infty [ \gamma(k\Delta_n)\gamma((k-p+q)\Delta_n) + \gamma((k+q)\Delta_n)\gamma((k-p)\Delta_n)].$$

In particular, focussing on the joint distribution of the sample variance and first lag autocovariance, it holds that $\sqrt{N}(\hat{\gamma}(0) - \gamma(0), \hat{\gamma}(\Delta_n) -\gamma(\Delta_n))$, is asymptotically normal with mean $0$, and covariance matrix $V \in \mathbb{R}^{2\times2}$ given by

\begin{align*}
\begin{split}
V_{00} &= \sum_{k = -\infty}^{\infty} 2 \gamma(k\Delta_n)^2, \\
V_{10} = V_{01} &=  \sum_{k = -\infty}^{\infty} 2 \gamma(k\Delta_n)\gamma((k-1)\Delta_n), \\
V_{11} &= \sum_{k = -\infty}^{\infty} [\gamma(k\Delta_n)^2 + \gamma((k+1)\Delta_n)\gamma((k-1)\Delta_n)].\\
\end{split}
\end{align*} 

Using the above results, we can determine the following limit theorem for the asymptotic distribution of $\tau_n^N$.

\begin{proposition}
	\label{prop:clt_scale_factor}
Let $Y$ be a $\mathcal{BSS}$ process and let $X$ be the Gaussian core of $Y$. Let $\tau_n^N$ be defined as above. Then for fixed $\Delta_n$, 

$$ \sqrt{N} \left( \tau_n^N - \tau_n \right) \xrightarrow{d} N\left(0,\frac{v}{4\tau_n^2}\right)$$
where the quantity $v$ is given by $v = 4V_{00} - 8V_{01} + 4V_{11}$ with $V_{00},V_{01}$ and $V_{11}$ as defined above.
\end{proposition}

An immediate consequence of this result is that $\tau_n^N$ converges in probability to $\tau_n$, since $\tau_n^N - \tau_n \xrightarrow{d} 0$ and convergence in distribution to a constant implies convergence in probability to that constant. $\tau_n^N$ thus exhibits long-span consistency in the sense of Equation \ref{eq:long_span_consistency}, however, in all non-trivial cases, we are not able to directly observe the Gaussian core $X$, and only have access to the observations of the $\mathcal{BSS}$ process $Y$. Thus $\tau_n^N$ is not truly feasible. 

To address this issue, we define the analogous scale factor based on the second order structure function for $Y$, 

$$\tau_n^Y := \sqrt{\mathbb{E}[|\Delta_1^n Y|^2]}$$
and we may define the corresponding empirical estimator for $\tau_n^Y$, based on $N$ observations of $Y$ at frequency $\Delta_n$ as 

$$\tau_n^{Y,N} := \sqrt{\frac{1}{N} \sum_{i=1}^{N} (\Delta_i^n Y)^2}.$$

Then we may establish the following weak law for the convergence of $\tau_n^{Y,N}.$

\begin{proposition}
	\label{prop:feasible_scale_factor_convergence}
Assume that $\sigma_{s}$ is independent of $W_{s}$, that $\sup_{s\in\mathbb{R}}\mathbb{E}[\sigma_{s}^{4}]<A_{1}$, that $\mathbb{E}[\sigma_{s}^{2}]$ is constant and smaller than $A_{2}$, that $Cov(\sigma_{s}^{2},\sigma_{r}^{2})\leq A_{3}c^{-A_{4}|s-r|}$, where $c>1$, for every $s,r\in\mathbb{R}$, and that $g(x)<A_{5}e^{-A_{6}x}$ for $x>K$, where $K\in\mathbb{R}_{+}$. Further, assume that $A_{6}>\frac{1}{2}A_{4}$. (in other words that $g$ decays sufficiently fast), that $\|g\|^2_{L^{2}}<A_{7}$ and that $\mathbb{E}(\Delta_1^{n} Y)^{2}\geq \tilde{A}_{n}$. Let
	\begin{equation*}
	D_{n}:=(2\wedge A_{4})\Delta_{n}(i-j),
	\end{equation*}
	\begin{equation*}
	C_{n}:=4A_{7}\left(A_{1}e^{2\Delta_{n}}\left(1+e^{2K}A_{7} \right) +\frac{2A_{5}A_{3}}{2A_{6}-A_{4}}\right), \quad \text{and}
	\end{equation*}
	\begin{equation*}
	E_{n}:=\frac{3\sqrt{C_{n}}}{\tilde{A}_{n}(1-e^{-D_{n}})}
	\end{equation*}
	Then for fixed $n$, we have
	\begin{align*}
	\tau_n^{Y,N}\xrightarrow{\mathbb{P}} \tau_n^Y,
	\end{align*}
	as $N\to \infty$, with $N\geq E_{n}^{2}n^{1+2\delta}$ for some $\delta > 0$, where $E_{n}<\frac{C'}{1-e^{-\Delta_{n}(2\wedge A_{4})}}$ and $C'$ is an explicit constant independent of $n$.
\end{proposition}

Since $\tau_n^{Y,N}$ depends only on the observations, we may now establish a truly feasible weak law for the convergence of the realised power variation as follows.

\begin{proposition}
	\label{prop:feasible_weak_law}
Assume that the conditions for the kernel $g$ are met, and let $\tau_n^{Y,N}$ be the non-parametric estimator based on $N$ observations of $Y$. Then 

$$\plim_{n \rightarrow \infty} \plim_{N \rightarrow \infty} \Delta_n (\tau_n^{Y,N})^{-p} V(Y,p;\Delta_n)(t) = \mathbb{E}[\sigma_0^2]^{-p/2}V(Y,p)(t) = m_p \mathbb{E}[\sigma_0^2]^{-p/2}\int_0^t |\sigma_s|^p \ \mathrm{d}s,$$
where again $\plim$ denotes that the limits are taken in probability. 
\end{proposition}

\begin{proof}
This is a direct consequence of the convergence in probability of $\tau_n^{Y,N} / \sqrt{\mathbb{E}[\sigma_0^2]}$ to $ \tau_n$ and an application of Proposition \ref{prop:feasible_scale_factor_convergence}.
\end{proof}

Now that we have established these variations of the weak law for the power variation, and arrived at a number of consistent estimators for feasible inference, let us consider applications to a specific choice of kernel, the gamma kernel, and we can see how each method of estimating $\tau_n$ can be used in this case.

\subsection{Scale factor for the gamma kernel}

For the gamma kernel $g(x) = x^\alpha e^{-\lambda x}$, we can obtain a closed form expression for the scale factor directly from the definition by first noting that $\gamma_X(0) = 2^{-2\alpha - 1} \lambda^{-2\alpha - 1} \Gamma(2\alpha + 1)$ and that $\gamma_X(h) = \frac{1}{2} \lambda^{-2\alpha - 1} \bar{K}_{\alpha + \frac{1}{2}}(\lambda h)$ for $h > 0$, where $K_\nu$ is the modified Bessel function of the third kind, and we define $ \bar{K}_\nu (x) = x^\nu K_\nu(x)$ \cite{barndorff2012notes}. Combining these two results, straightforward manipulations give the scale factor as 

$$ \tau_n =  \lambda^{-\alpha - \frac{1}{2}}  \left\{ \frac{\Gamma(\alpha + 1)}{\Gamma(\frac{1}{2})} \left( \Gamma\left(\alpha + \tfrac{1}{2}\right) - 2^{-\alpha + \frac{1}{2}} \bar{K}_{\alpha + \frac{1}{2}} (\lambda \Delta_n) \right) \right\}^{\frac{1}{2}}.$$

Since we have access to a closed form for $\tau_n$, we can use this in the simulation study to assess performance of the infeasible weak law, under the assumption of known parameters $\alpha$ and $\lambda$, and also our plug-in estimators after parameter estimation. We can also derive the asymptotic behaviour of $\tau_n$, as it has been shown \cite{barndorff2012notes} that the asymptotic behaviour of the Bessel function near zero is 

$$\bar{K}_\nu(x) = 2^{\nu -1} \Gamma(\nu) \left\{ 1 - 2^{-2\nu} \frac{\Gamma(1 - \nu)}{\Gamma(1 + \nu)}x^{2\nu} + \mathcal{O}(x^2) \right\}  \qquad \text{as} \ x \rightarrow 0^+.$$ 

Applying this to our case gives us the asymptotic result for $\tau_n^2$ that

\begin{align*}
\begin{split}
\tau_n^2 &= \frac{\Gamma\left(\alpha + 1\right) \Gamma\left(\alpha + \tfrac{1}{2}\right) \Gamma\left(\tfrac{1}{2} - \alpha \right)}{ \Gamma\left(\tfrac{1}{2}\right) \Gamma\left(\alpha + \tfrac{3}{2}\right)} \Delta_n^{2\alpha + 1} 2^{-\alpha - 1} + \mathcal{O}(\Delta_n^2) \\
&= 2^{-4\alpha - 1}\frac{\Gamma\left(2\alpha + 1\right)\Gamma\left( \frac{1}{2} - \alpha \right)}{\Gamma\left(\alpha + \tfrac{3}{2}\right)} \Delta_n^{2\alpha + 1}  + \mathcal{O}(\Delta_n^2)
\end{split}
\end{align*}
and thus for small $\Delta_n$, we may approximate $\tau_n$ by

$$\tilde{\tau}_n = 2^{-2\alpha - \frac{1}{2}} \left( \frac{\Gamma\left(2\alpha + 1\right)\Gamma\left( \frac{1}{2} - \alpha \right)}{\Gamma\left(\alpha + \tfrac{3}{2}\right)} \right)^{\frac{1}{2}} \Delta_n^{\alpha + \frac{1}{2}}.$$

Note that the asymptotic behaviour of the scale factor depends only on the smoothness parameter $\alpha$ and not the other parameter $\lambda$, which is of course because $\lambda$ controls the long term memory of the $\mathcal{BSS}$ process, and at the fine scale only $\alpha$ affects the behaviour.

\subsection{Central limit theorem for power variations}

Having established that the realised power variation converges weakly to the integrated volatility process, \cite{corcuera2013asymptotic} establishes the following central limit theorem for this convergence. 

\begin{theorem}
	\label{thm:infeasible_clt}
Assume that the conditions for the kernel function $g$ are satisfied and the process $\sigma$ is H{\"o}lder continuous of order $\gamma \in (0,1)$ for some $\gamma(p\wedge1)> 1/2.$ Suppose further that $\alpha \in (-\frac{1}{2},0)$. Then the following stable convergence holds

$$\Delta_n^{-1/2} \left(\Delta_n \tau_n^{-p} V(Y,p;\Delta_n)(t) - V(Y,p)(t) \right) \xrightarrow{st} \Lambda_p \int_0^t |\sigma_s|^p \ \mathrm{d}B_s$$
on the space $\mathbb{D}([0,T])$ equipped with a uniform topology, where $B$ is a Brownian motion defined on an extension of the original probability space $(\Omega, \mathcal{F}, \mathbb{P})$ and is independent of $\mathcal{F}.$ The factor $\Lambda_p$ is defined as 

$$\Lambda_p^2 := \lim_{n\rightarrow\infty} \Delta_n^{-1} \Var \left( \Delta_n^{1-pH} V(B^H, p; \Delta_n)(1) \right),$$
where $B^H$ denotes a fractional Brownian motion with Hurst parameter $H = \alpha + 1/2.$
\end{theorem}

We can see that the limiting process is a mixed Gaussian process, with zero mean and conditional variance given by 

$$ \mathbb{E}\left[ \left( \Lambda_p \int_0^t |\sigma_s|^p \ \mathrm{d}B_s \right)^2 \right] =   \Lambda_p^2 \int_0^t |\sigma_s|^{2p} \ \mathrm{d}s.$$

Thus, a consequence of the central limit theorem for power variations is that, in the limit as $n \rightarrow \infty$, we have 

$$\Delta_n^{-1/2} \frac{\left(\Delta_n \tau_n^{-p} V(Y,p;\Delta_n)(t) - V(Y,p)(t) \right)}{\Lambda_p \sqrt{ \int_0^t |\sigma_s|^{2p} \ \mathrm{d}s}} \dot{\sim} N(0,1).$$
where $\dot{\sim}$ denotes asymptotic normality. As with the weak law, this limit theorem is infeasible in the sense that it relies on knowledge of the scaling factors $\tau_n$ and $\Lambda_p$, which are both dependent on the underlying Gaussian core of the $\mathcal{BSS}$ process, and also because the denominator relies on $\sigma$. So in an analagous manner to before, we may construct a feasible limit theorem as follows. 

\begin{theorem}
	\label{thm:feasible_clt}
Assume as before that the conditions for the kernel function $g$ are satisfied and the process $\sigma$ is H{\"o}lder continuous of order $\gamma \in (0,1)$ for some $\gamma(p\wedge1)> 1/2.$ Suppose further that $\alpha \in (-\frac{1}{2},0)$. Let $(N^{*}_{n})_{n\in\mathbb{N}}$ be such that $N^{*}_{n}=E_{n}^{2}n^{1+2\delta}$. Then, for every $(N_{n})_{n\in\mathbb{N}}$ with $N_{n}\geq N^{*}_{n}$ and $n\in\mathbb{N}$, we have
	\begin{align*}
	\Delta_n^{-1/2}\left(\Delta_n (\tau_n^{Y,N_n})^{-p}V(Y,p;\Delta_n)(t) -[\mathbb{E}(\sigma_0^2)]^{-p/2}V(Y,p)(t)\right) \xrightarrow{st} \Lambda_p [\mathbb{E}(\sigma_0^2)]^{-p/2}\int_0^t |\sigma_s|^p \ \mathrm{d}B_s,
	\end{align*}
	for every $t\geq0$ where $B$ denotes a Brownian motion defined on an extension of the original probability space $(\Omega, \mathcal{F}, \mathbb{P})$, which is independent of $\mathcal{F}$, and, as before,
	\begin{align*}
		\Lambda_p^2 := \lim_{n \to \infty}\Delta_n^{-1}\textnormal{Var}\left(\Delta_n^{1-pH}V(B^H,p;\Delta_n)(1)\right),
	\end{align*}
	where $B^H$ denotes a fractional Brownian motion with Hurst parameter $H = \alpha + 1/2$. 
\end{theorem}

So now we can replace this quantity with an estimator, using the weak law previously established. The non-random quantity $\Lambda_p$ is expressible as an infinite sum \cite{corcuera2013asymptotic} and hence we may approximate this by simply truncating the sum. Then, since we have established that 

$$\lim_{n \rightarrow \infty} \lim_{N \rightarrow \infty} \Delta_n (\tau_n^{Y,N})^{-2p} V(Y,2p;\Delta_n)(t) = \mathbb{E}[\sigma_0^2]^{-p}V(Y,2p)(t) = m_{2p} \mathbb{E}[\sigma_0^2]^{-p}\int_0^t |\sigma_s|^{2p} \ \mathrm{d}s,$$
we can use this estimate, and separately approximate $\Lambda_p$, to formulate an approximate central limit theorem as 

$$\Delta_n^{-1/2} \frac{\left(\Delta_n (\tau_n^Y)^{-p} V(Y,p;\Delta_n)(t) - \mathbb{E}[\sigma_0^2]^{-p/2}V(Y,p)(t) \right)}{\widehat{\Lambda}_p \sqrt{ m_{2p}^{-1}(\tau_n^{Y,N})^{-2p} V(Y,2p;\Delta_n)(t)}} \sim N(0,1).$$

We could loosely consider this to be a `semifeasible' limit theorem, in the sense that the only quantity that is not estimated directly from the data is $\tau_n^Y$. Barndorff-Nielsen et al. \cite{barndorff2014assessing} have shown that where we may estimate $\alpha$ in a consistent way, then we may replace $\tau_n$ with an estimator $\hat{\tau}_n$ and then the limit is changed, both in terms of limiting distribution and in terms of the rate of convergence, with the estimation of $\alpha$ slowing the rate of convergence. We may seek to go a step further and ask whether we can replace $\tau_n^Y$ with the estimator $\tau_n^{Y,N}$, assess whether the limit theorem still holds. We may consider the distribution of

$$\Delta_n^{-1/2} \frac{\left(\Delta_n (\tau_n^{Y,N})^{-p} V(Y,p;\Delta_n)(t) - \mathbb{E}[\sigma_0^2]^{-p/2}V(Y,p)(t) \right)}{\widehat{\Lambda}_p \sqrt{ m_{2p}^{-1}(\tau_n^{Y,N})^{-2p} V(Y,2p;\Delta_n)(t)}}$$
in the limit as $n \rightarrow \infty, N \rightarrow \infty$, and consider this to be a `feasible' limit theorem, if the convergence still holds. 

The limit theorem is important as it enables us to go beyond point estimations and obtain asymptotic confidence intervals for the integrated volatility. For $a \in (0,1)$, a $(1-a) \times 100 \%$ asymptotic confidence interval for $V(Y,p)(t)$ in the semifeasible case is given by

\begin{equation*}
\begin{split}
\Bigg(\Delta_n \tau_n^{-p} V(Y,p;\Delta_n)(t) &- z_{1-a/2} \Delta_n^{1/2} \widehat{\Lambda}_p \sqrt{ m_{2p}^{-1}(\tau_n^{Y,N})^{-2p} V(Y,2p;\Delta_n)(t)}, \\
&\Delta_n \tau_n^{-p} V(Y,p;\Delta_n)(t) + z_{1-a/2} \Delta_n^{1/2} \widehat{\Lambda}_p \sqrt{ m_{2p}^{-1}(\tau_n^{Y,N})^{-2p} V(Y,2p;\Delta_n)(t)}\Bigg).
\end{split}
\end{equation*}

The question for the simulation study will be whether the corresponding feasible interval, 

\begin{equation*}
\begin{split}
\Bigg(\Delta_n (\tau_n^{Y,N})^{-p} V(Y,p;\Delta_n)(t) &- z_{1-a/2} \Delta_n^{1/2} \widehat{\Lambda}_p \sqrt{ m_{2p}^{-1}(\tau_n^{Y,N})^{-2p} V(Y,2p;\Delta_n)(t)}, \\
&\Delta_n (\tau_n^{Y,N})^{-p} V(Y,p;\Delta_n)(t) + z_{1-a/2} \Delta_n^{1/2} \widehat{\Lambda}_p \sqrt{ m_{2p}^{-1}(\tau_n^{Y,N})^{-2p} V(Y,2p;\Delta_n)(t)}\Bigg),
\end{split}
\end{equation*}
is an appropriate approximate confidence interval for feasible inference on the integrated volatility process.

\section{Simulation study}

To test the convergence of the asymptotic theory in practice, we employ a simulation study, whereby a number of $\mathcal{BSS}$ sample paths are simulated, and subsampled to provide a number of paths with varying frequencies. Then the finite sample behaviour of the various asymptotic results can be tested, to assess convergence in both in-fill and long span. In this chapter we outline the methods and details such as parameter choices involved in the simulation study. The approach is an implementation of the \text{hybid scheme method} described in \cite{bennedsen2014discretization} and \cite{bennedsen2017hybrid}. The code implementation for the simulation of the processes, fitting of $\mathcal{BSS}$ process to data and estimation of the accumulated volatility have been arranged into new \texttt{R} package, \texttt{BSS}, available on CRAN. 

\subsection{Volatility process}
The volatility process chosen is an exponentiated Ornstein-Uhlenbeck process,

$$\sigma_t = \exp( \beta v_t)$$
where $\beta \in \mathbb{R}$ and $v_t$ satisfies the stochastic differential equation 
$$dv_t = - \theta v_t \ dt + dB_t$$
for $\theta > 0$, where $B$ is a Brownian motion which is independent of $W$. We don't actually need to specify that $\sigma$ is independent of $W$, as the theoretical convergence results do not depend this assumption, and we could have some (possibly stochastic) correlation $\rho_t$ between the two and then use the decomposition 

$$B_t = \rho_t W_t + \sqrt{1 - \rho_t^2} W^{\bot}_t, $$
where $W^{\bot}$ is a Brownian motion independent of $W$. However, in the simulation study we will assume independence for simplicity, and so that $\sigma$ can be simulated separately from $W$. This can be simulated directly, by using an Euler scheme to discretise the differential equation into 

$$ v_{i\Delta_n} = (1- \Delta_n \theta)v_{(i-1)\Delta_n} + \sqrt{\Delta_n} B_i $$
where $B_i$ is a draw from a standard normal distribution for $i = -N_n, -N_n + 1, ... , nT - 1.$ The volatility process is then generated by taking $\sigma_{i\Delta_n} = \exp( \beta v_{i\Delta_n})$. Since we want the volatility process to be in the stationary regime, we initialise $v_{-N_n \Delta_n}$ from the stationary distribution $N(0, 1/2\theta)$ and then propagate the rest of the sample path.

\subsection{Parameter choices and simulation procedure}
In the simulation study, we simulate $M = 5000$ sample paths at a frequency of $\Delta_n = 1/25000$ with $T = 1$. This choice of $T$ and $\Delta_n$ reflects approximate simulation of the process over a trading day, with observations taken every second (true value $25200$ observations), so the paths therefore could reflect the price of a stock over a market day. So throughout the discussion of the results, we will take $\Delta_n = 1/25000$ to represent $1s$, and hence $\Delta_n = 60/25000$ to represent $1m$, and so on, and use the two representations for frequency interchangeably. 

For the volatility process we set $\theta = 2$ and $\beta = 0.125$, so that the volatility does not fluctuate too rapidly. For the $\mathcal{BSS}$ process, we choose a gamma kernel with smoothness parameter $\alpha = -0.2$, and exponent parameter $\lambda = 1$. The real parameter of interest is the smoothness parameter $\alpha$, and this choice takes us into the non-semi-martingale case, and also with the kernel diverging at zero. For each simulation, the $\mathcal{BSS}$ sample path, $Y_{i\Delta_n}$ for $i = 0,1,...,25000$ is saved, as well as the Gaussian core $X_{i\Delta_n}$, found by using the same Brownian increments and employing the hybrid scheme with the volatility set to be identically one, and the volatility process $\sigma_{i\Delta_n}$ for $i = 0,1,...,25000.$

To simulate $\mathcal{BSS}$ paths at different frequencies, these sample paths are simply subsampled, so that for example a path at frequency $\Delta_n = 1/12500$ is obtained by simply taking every other element of the original path. All calculations are done exactly where possible, and the integrated volatility process was approximated by discretising the integral from the simulated paths,

$$ \int_0^t |\sigma_s|^p \ \mathrm{d}s \approx \Delta_n \sum_{i = 0}^{\lfloor t/\Delta_n \rfloor} |\sigma_{i\Delta_n}|^p .$$

\section{Results of the simulation study}
\subsection{Non-parametric estimation of the scale factor}
We first assess the convergence of the non-parametric estimator for the scale factor, in particular whether the convergence established in Proposition \ref{prop:clt_scale_factor} holds. We approximate $V_{00}, V_{01}$ and $V_{11}$ by truncating the sums, and hence calculate the variance of the limiting distribution.

Figure \ref{fig:density_tau_N} shows density plots of 

$$ \sqrt{N} \frac{ \left( \tau_n^N - \tau_n \right)}{\sqrt{v}/2\tau_n},$$ 
for $\Delta_n = 1/25000$ and $N = 100, 10000, 25000$, overlaid on standard $N(0,1)$ densities to show the limit theorem taking effect. We can see that the limit theorem is shown very rapidly, with the centred and scaled distribution resembling the standard Gaussian with only $N=100$ observations. Of course, this non-parametric estimator is infeasible in that it relies on the Gaussian core. So we also consider the accuracy of the feasible non-parametric estimator, $\tau_n^{Y,N}$. Figure \ref{fig:density_tau_YN} shows plots of 

$$ \sqrt{N} \frac{ \left(\tau_n^{Y,N} - \exp\left(\beta^2/2\theta\right)\tau_n \right)}{v^Y},$$ 
for $\Delta_n = 1/25000$ and $N = 100, 1000, 25000$, overlaid on standard $N(0,1)$ densities. What is clear is that the distribution does not converge towards a standard Gaussian as $N$ increases, and indeed the variance here increases with $N$, suggesting that the rate of convergence is slower than $\sqrt{N}$. Comparing the standard deviations of $\tau_n^{Y,N}$ for different values of $N$ suggests that the rate of convergence is approximately $N^{0.14}$, indicating experimentally that $\tau_n^{Y,N}$ is indeed a consistent estimator for $\sqrt{\mathbb{E}[\sigma_0^2]} \tau_n,$ but that the rate of convergence is slower than for the infeasible non-parametric estimator $\tau_n^N.$ 

\begin{figure}[!b]
\centering
\includegraphics[width=16cm]{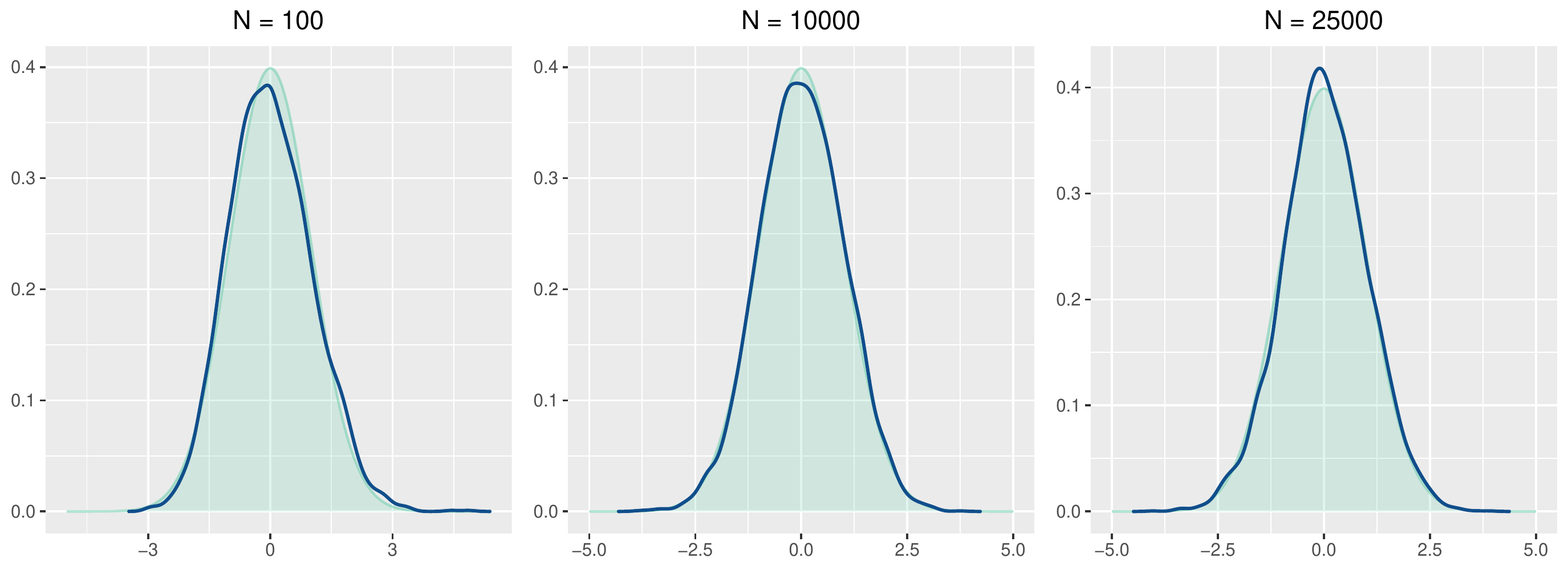}
\caption{Centred and scaled density plots for the non-parametric estimator $\tau_n^N$ using $\Delta_n = 1/25000$ for $N=100,10000, 25000$, overlaid on $N(0,1)$.}
\label{fig:density_tau_N}
\end{figure}

\begin{figure}[!b]
\centering
\includegraphics[width=16cm]{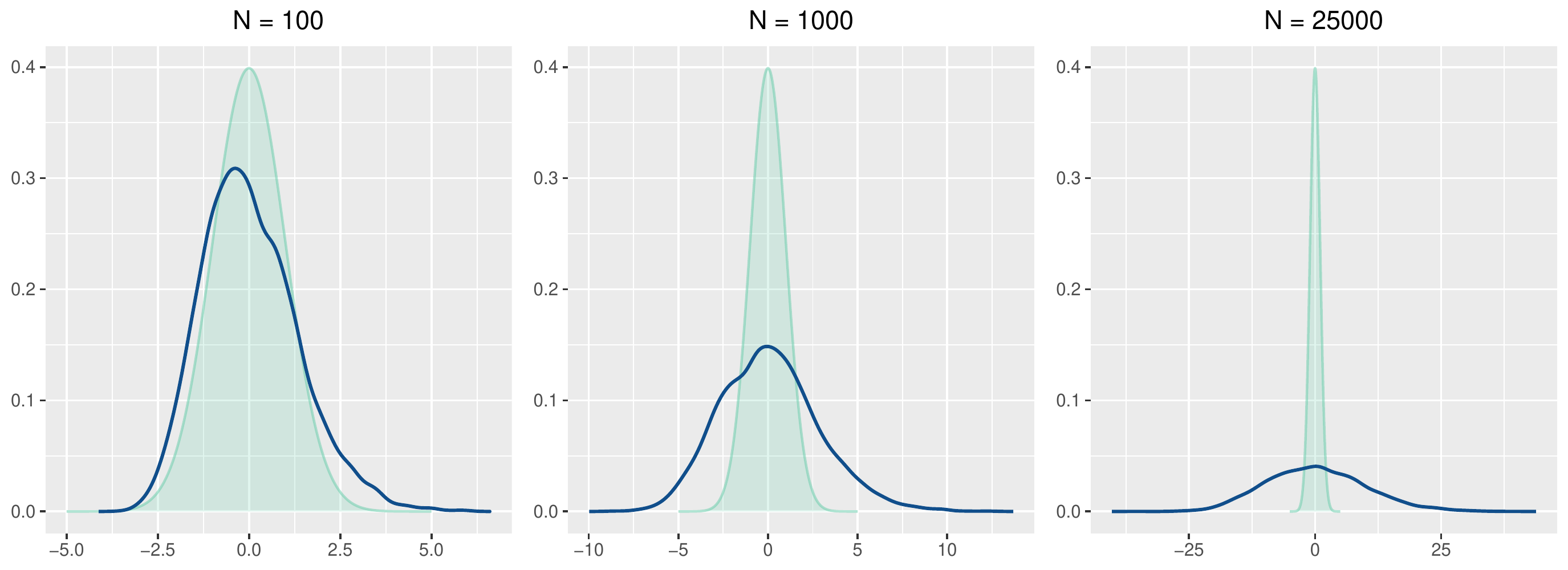}
\caption{Centred and scaled density plots for the non-parametric estimator $\tau_n^{Y,N}$ using $\Delta_n = 1/25000$ for $N=100,1000, 25000$, overlaid with $N(0,1)$.}
\label{fig:density_tau_YN}
\end{figure}

\subsection{Weak law}
In Table \ref{table:rmse}. we compare the root mean square error of a the weak law using a number of estimators of the integrated volatility, each defined by the estimator of the scale factor used. We compare the infeasible oracle estimate using $\tau_n$, with the infeasible  estimate using the asymptotic expression $\tilde{\tau}_n$, two parametric estimators which can be shown to be consistent: $\tau_n^{ACF}$ which uses least squares regression on the autocorrelation function of the paths to estimate $\alpha$ and puts that into the asymptotic expression for the scale factor (see \cite{bennedsen2020semiparametric} for more on this approach), and $\tau_n^{COF}$ which estimates $\alpha$ using the change of frequency method described in \cite{corcuera2013asymptotic}. We present results for the case $p = 2$, the integrated squared volatilty.  

The tables indicate that from the four feasible estimators, the non-parametric estimate performs the best, at all frequencies. We see that at high frequencies, the performance of the estimators based on $\tau_n^{ACF}$ and $\tau_n^{COF}$ are very similar, indicating that they estimate $\alpha$ to a similar level of accuracy, but at lower frequencies, the error of $\tau_n^{ACF}$ is far higher. Where an error was extremely large, we simply indicate that it is over $10^3$. Both parametric estimators become very innacurate at lower frequencies, with the ACF method performing poorly. This is due to slow decay of the ACF which causes the sample ACF to differ from the true ACF unless the number of observations is very large, in turn leading to large parameter estimation errors. On the other hand, the COF method only relies on the local behaviour of the process so is less affected by data not covering long time spans. 

We also witness counter-intuitive behaviour for the non-parametric estimator, that although at high frequencies it performs worse than the infeasible estimator, it actually performs better, in terms of RMSE, at frequencies lower than $5m$, which defies the intuition that the infeasible should be the best performing at any frequency, as it uses the true scale factor. In this case, the scale factor and realised power variation were both estimated using all available data ($t=T$), and so actually cancel out with each other in the calculation, leading to a constant value across all paths. For fixed $t$, we can improve the estimates for $V(Y, 2)(t)$ as new data is available by updating the estimates of $\tau_n^{N,Y}$ using all observations up to some further time $T > t$.

In general, for the simulation study, the non-parametric estimator provides the lowest error in estimating the integrated volatility process across all feasible estimators. 

\begin{table}[!h]
\centering
\begin{tabular}{c|cccccc}
\multicolumn{1}{c|}{\textbf{\textit{RMSE}}} & \multicolumn{6}{c}{$\Delta_n$} \\ 
\hline
Scale factor used & 1s  & 5s & 30s  & 1m & 5m  & 30m\\
\hline
$\tau_n$ 		&	.0096	&	.0213	&	.0527	&	.0751	&	.1661	&	.4042	\\
$\tilde{\tau}_n$	&	.0096	&	.0213	&	.0527	&	.0751	&	.1660	&	.4023 	\\
$\tau_n^{ACF}$		&	.3735	&	.5042	&	1.323	&	2.097	&	$>10^3$	&	$>10^3$	\\
$\tau_n^{COF}$		&	.2347	&	.4935	&	1.439	&	2.463	&	10.21	&	7.952			\\
$\tau_n^{Y,N}$		&	.0949	&	.0949	&	.0949	&	.0950	&	.0951	&	.1156	\\
\end{tabular}
\caption{Root mean square errors in estimating the integrated squared volatility process, comparing different scale factors.}
\label{table:rmse}
\end{table}

\subsection{Infeasible limit theorem results}
As a benchmark, we first test the convergence of the infeasible limit theorem, by comparing the distribution of the centred and scaled errors, 

$$ \Delta_n^{-1/2} \frac{\left(\Delta_n \tau_n^{-2} V(Y,2;\Delta_n)(t) - V(Y,2)(t) \right)}{\Lambda_2 \sqrt{ \int_0^t \sigma_s^{4} \ \mathrm{d}s}},$$
with the standard $N(0,1)$. Again we focus on $p=2$ as the other cases are similar. The constant $\Lambda_2$ was calculated by truncating the sum described in \cite{corcuera2013asymptotic}, and the integral $\int_0^t \sigma_s^{4} \ \mathrm{d}s$ was approximated using a Riemann sum from the simulated volatility paths. Figure \ref{fig:infeasible_clt_density} shows density plots for the centred and scaled errors, compared with a $N(0,1)$ distribution. We can observe that the limit theorem holds well for $\Delta_n = 1s$, and also reasonably well for $\Delta_n = 5m$, but not so well at the lower frequency. Therefore, as a rule of thumb, we suggest that we start to observe the limit theorem taking effect at around the $5m$ frequency. 

\begin{figure}[!t]
\includegraphics[width=16cm]{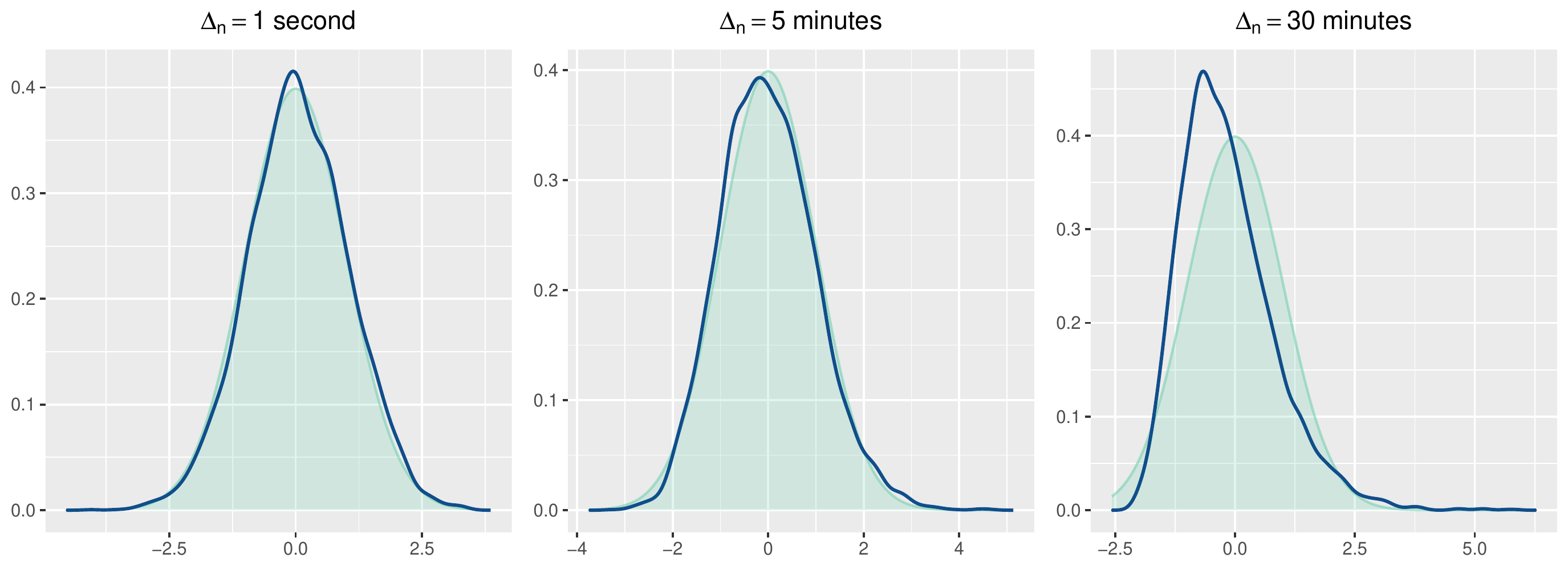}
\caption{Density plots for the centred and scaled errors for estimating $V(Y,2)(t)$ in the infeasible case, for $\Delta_n = 1s, 5m, 30m$.} 
\label{fig:infeasible_clt_density}
\centering
\end{figure}

\subsection{Semifeasible limit theorem results}
We now assess the convergence of a semifeasible limit, where the denominator has been replaced by the estimate based on the non-parametric estimator. So we may assess whether the following asymptotic distribution holds, again focussing on $p=2$, as $N \rightarrow \infty$, $n \rightarrow \infty$

$$\Delta_n^{-1/2} \frac{\left(\Delta_n (\tau_n^Y)^{-2} V(Y,2;\Delta_n)(t) - \mathbb{E}[\sigma_0^2]^{-1}V(Y,2)(t) \right)}{\widehat{\Lambda}_2 \sqrt{ m_{4}^{-1}(\tau_n^{Y,N})^{-4} V(Y,4;\Delta_n)(t)}} \sim N(0,1).$$

Figure \ref{fig:clt_semi_feasible_density} shows density plots for these semifeasible errors, compared with standard a $N(0,1)$ distribution. Here we see eventual convergence for $\Delta_n = 1s$, but not for $5m$ or $30m$, which demonstrate a far higher variance. Note that since these were calculated over the same sample paths of finite length, subsampled for the lower frequencies, we are increasing both the number of observations $N$, and frequency $n$, as we move from lower frequencies to higher frequencies, which is probably why there is such a noticeable improvement in the convergence - at $30m$ we have $n = 1800$, but only $N = 13$ observations, compared with $n = 25000$ and $N = 25000$ for the $1s$ case. We would assume that for fixed $\Delta_n$, with longer sample paths, the convergence would improve, so that the limit theorem would be observed at lower frequencies. Given the simulated data, the semifeasible limit theorem begins to be observed for frequencies around $1m$ and higher.

\begin{figure}[!t]
\includegraphics[width=16cm]{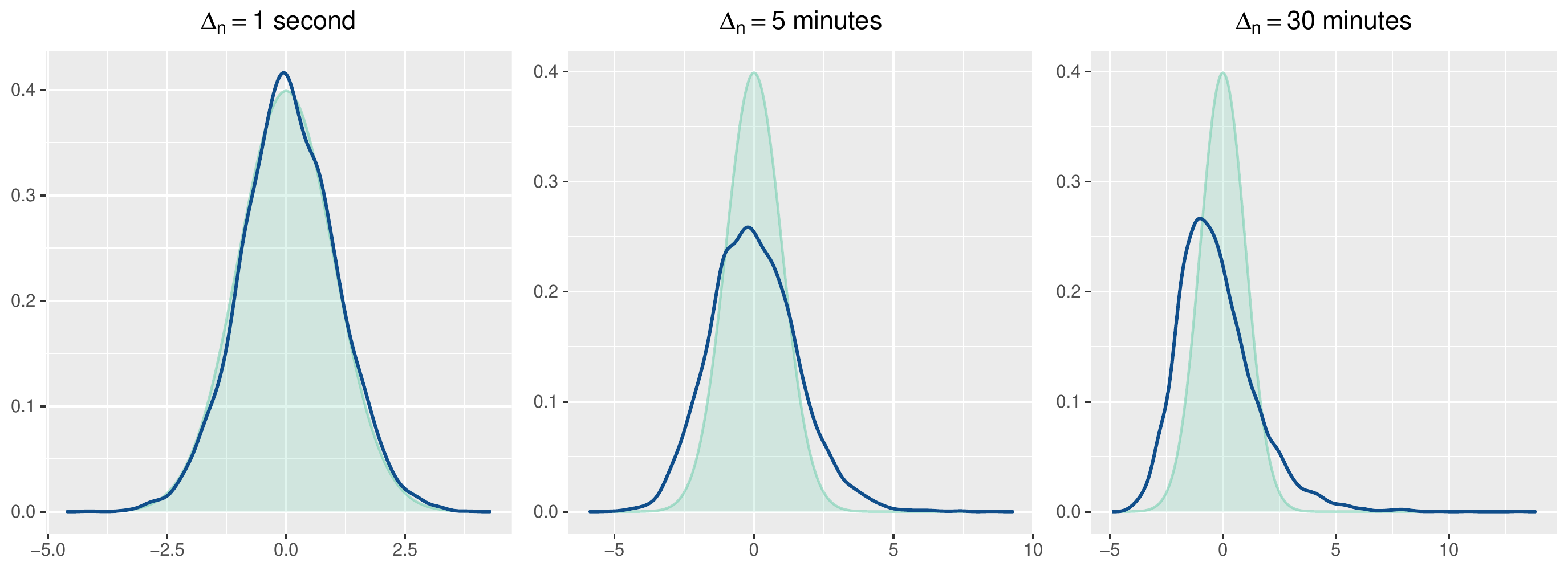}
\caption{Density plots for the centred and scaled errors for estimating $V(Y,2)(t)$ in the semifeasible case, for $\Delta_n = 1s, 5m, 30m$.} 
\label{fig:clt_semi_feasible_density}
\centering
\end{figure}

\subsection{Feasible limit theorem results}
We now consider the distribution of our feasible estimator errors, i.e. we look at 

$$\Delta_n^{-1/2} \frac{\left(\Delta_n (\tau_n^{Y,N})^{-2} V(Y,2;\Delta_n)(t) - \mathbb{E}[\sigma_0^2]^{-1}V(Y,2)(t) \right)}{\widehat{\Lambda}_2 \sqrt{ m_{4}^{-1}(\tau_n^{Y,N})^{-4} V(Y,4;\Delta_n)(t)}},$$
and again compare with a standard $N(0,1)$ distribution. As can be seen from Figure \ref{fig:non_parametric_clt_density}, the overall shape of the distribution appears approximately normal, but the variance is increasing as $\Delta_n$ decreases, directly mirroring the increasing variance witnessed in the estimation of $\tau_n^{Y,N}$. Thus we can arrive at a similar conclusion, that the convergence is at a rate which is slower than $\Delta_n^{1/2}$. Note that we have again restricted ourselves to $N = \lfloor T / \Delta_n \rfloor$, so that we are not considering the effect of infill or long span asymptotics independently, but are taking the limit of both simultaneously. This could account for some of the variation seen here, and further study could simulate longer sample paths so that we could isolate for example the effect of changing $\Delta_n$ for fixed, large $N$, or alternatively, isolate the effect of increasing $N$ for fixed $\Delta_n$.

\begin{figure}[!t]
\includegraphics[width=16cm]{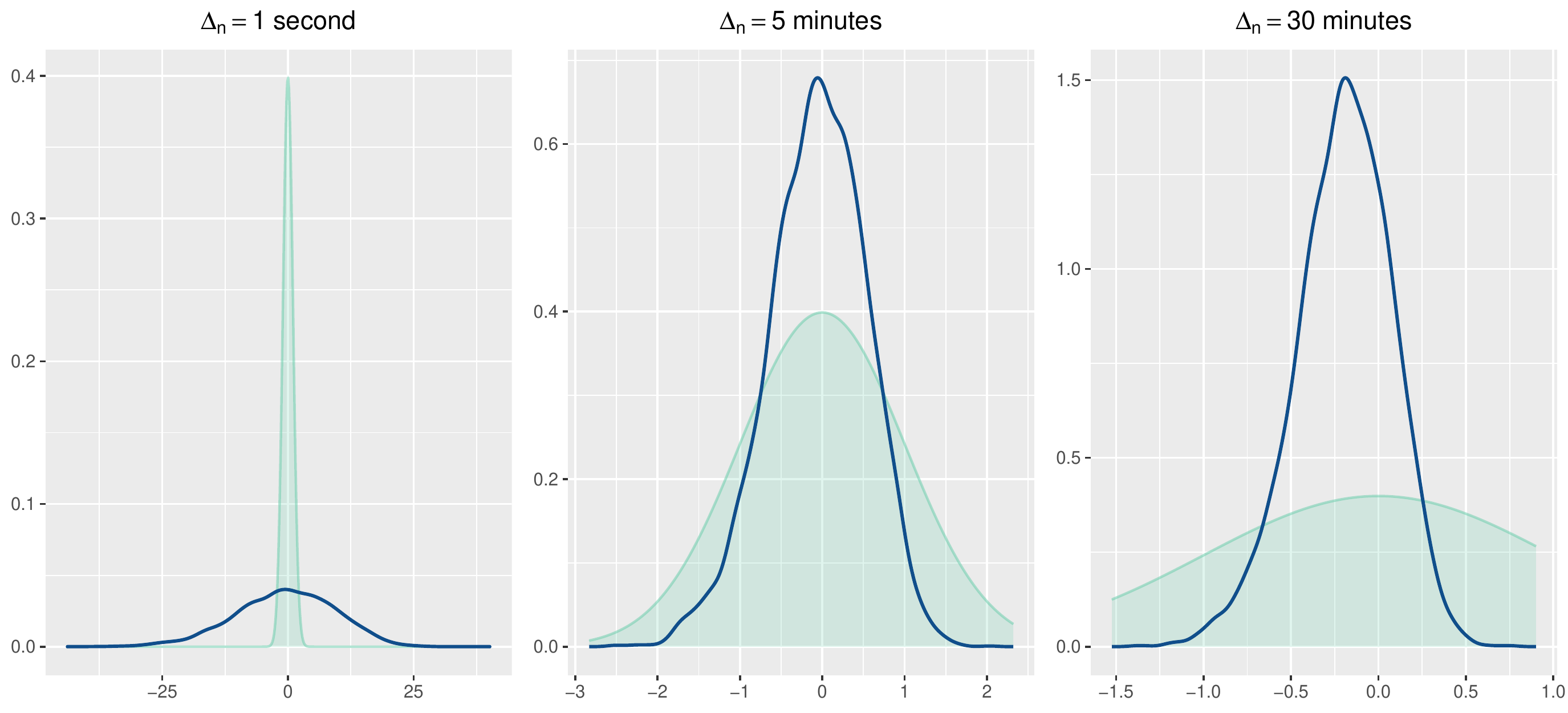}
\caption{Density plots for the centred and scaled errors for estimating $V(Y,2)(t)$ in the feasible case, for $\Delta_n = 1s, 5m, 30m$.} 
\label{fig:non_parametric_clt_density}
\centering
\end{figure}

\section{Application to turbulence data}
We apply the methodology developed above to perform inference on the underlying volatility of turbulence data. The data comprises observations made at \text{Brookhaven National Laboratory} (Long Island, NY), and consists of a time series of the main component of a turbulent velocity vector, measured at a fixed position in the atmospheric boundary layer using a hotwire anemometer, during an approximately 66 minutes long observation period at sampling frequency of 5 kHz (i.e. 5000 observations per second). A comprehensive account of the data has been given by \cite{dhruva2001experimental} and has already been studied using Brownian semistationary processes to  
test Kolmogorov's 5/3-law, that within the so-called \textit{inertial range} of frequencies, we should observe a smoothness parameter of approximately $-1/6$. As discussed in \cite{corcuera2013asymptotic}, a suitable inertial range for this dataset is $0.1Hz -  100Hz$, and thus we consider the performance of our estimators on time series subsampled from the original data to these frequencies. 

We first estimate the smoothness parameter $\alpha$ using the change-of-frequency method described in \cite{barndorff2014assessing} over a range of frequencies and compare this with the theoretical value from Kolmogorov's law. Figure \ref{fig:alpha_cof_freqs} shows that our estimates of $\alpha$ fit reasonably well with the theoretical law, especially in the middle range of frequencies, between $1.0Hz - 10Hz$. As the frequency increases, the estimate of $\alpha$ also increases, indicating that the paths appear smoother at very small timescales, a point noted in \cite{barndorff2014assessing}. Having estimated $\alpha$, we may then compare estimates of the scale factor $\tau_n$ derived through $\tau_n^{COF}$ and $\tau_n^{Y,N}$, comparing them with $\tilde{\tau}_n$ with $\alpha = -1/6$. Figure \ref{fig:tau_plot2} demonstrates the similarity between the two estimates over a range of the entire range of frequencies. 

We next compare the approaches to estimate the integrated squared volatility $V(Y, 2)(t)$. We estimate this at different frequencies using the non-parametric method and the change of frequency estimator for the scale factors. As witnessed in Figure \ref{fig:tau_plot2}, the change of frequency tends to underestimate the scale factor for high frequencies, since it overestimates $\alpha$. This in turn leads to an overestimate of the integrated squared volatility, as the more of the variation in the paths is attributed to the volatility process as rather than the kernel. 

Figure \ref{fig:int_vol} shows estimates for $V(Y, 2)(t)$ over an interval of one minute, where the data is subsampled at a frequency of $1.5Hz$. Since the only difference between the estimators is the scale factor, they all follow the same path, up to that factor, and we can see that at this frequency the change of frequency method is overestiamting the integrated volatility. Figure \ref{fig:acc_vol2} compares estimates for $V(Y, 2)(T)$ at time $T = 6$ minutes at different frequencies, where the scale factor estimates are calculated using the full paths. We can see again that the non-parametric method agrees well with Kolmogorov's law across the range of frequencies but that at frequencies above $10Hz$, estimating the smoothness parameter creates divergence in that estimate.
We can conclude from this brief data example that the non-parametric feasible estimator does give reliable and stable estimates of the integrated volatility, over a wide range of sampling frequencies, where the best parametric alternative only provides reliable estimates for a limited range.

\begin{figure}[t]
\begin{subfigure}{0.5\textwidth}
\includegraphics[scale = 0.5]{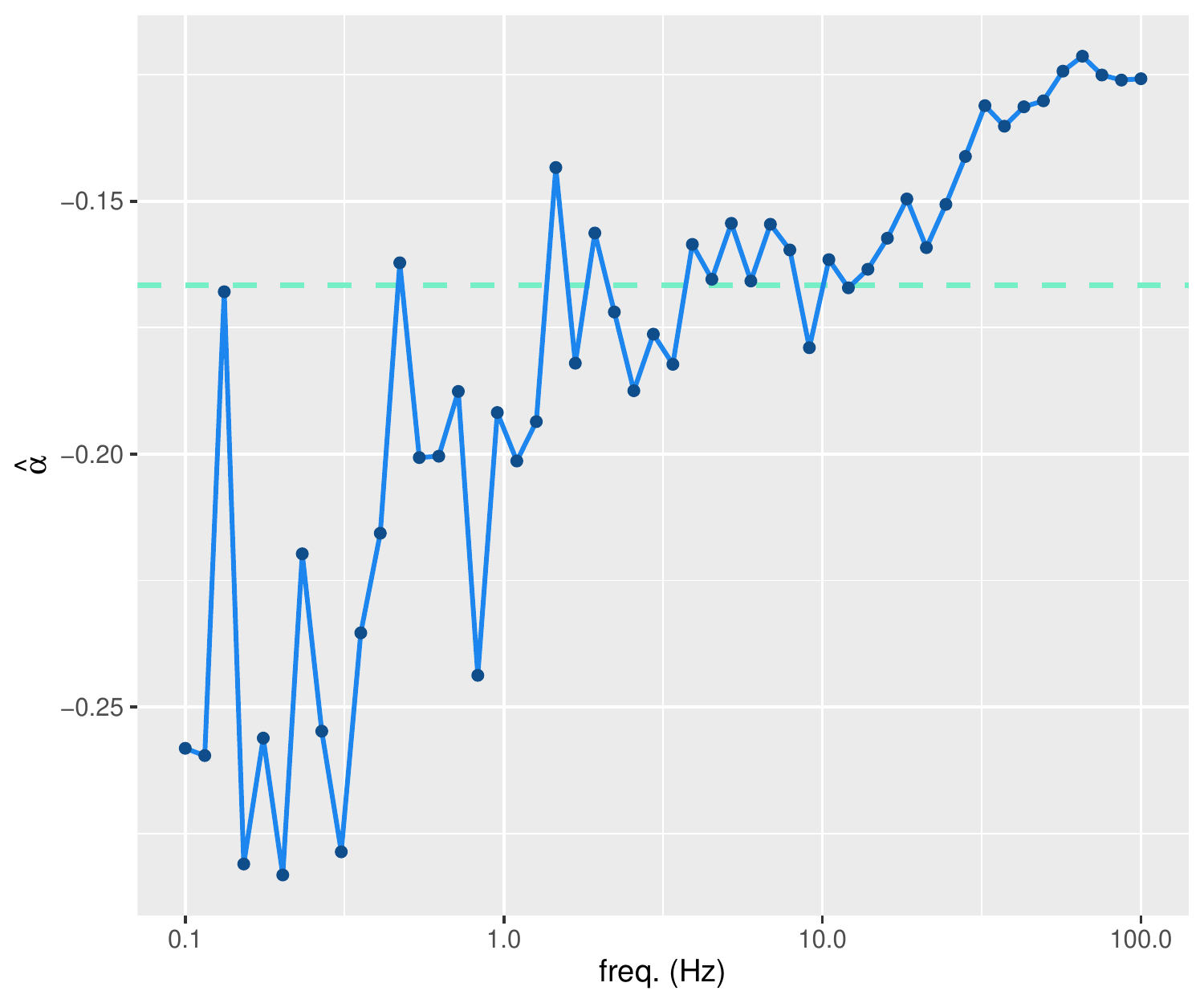}
\caption{}
\label{fig:alpha_cof_freqs}
\end{subfigure}
\begin{subfigure}{0.5\textwidth}
\includegraphics[scale = 0.5]{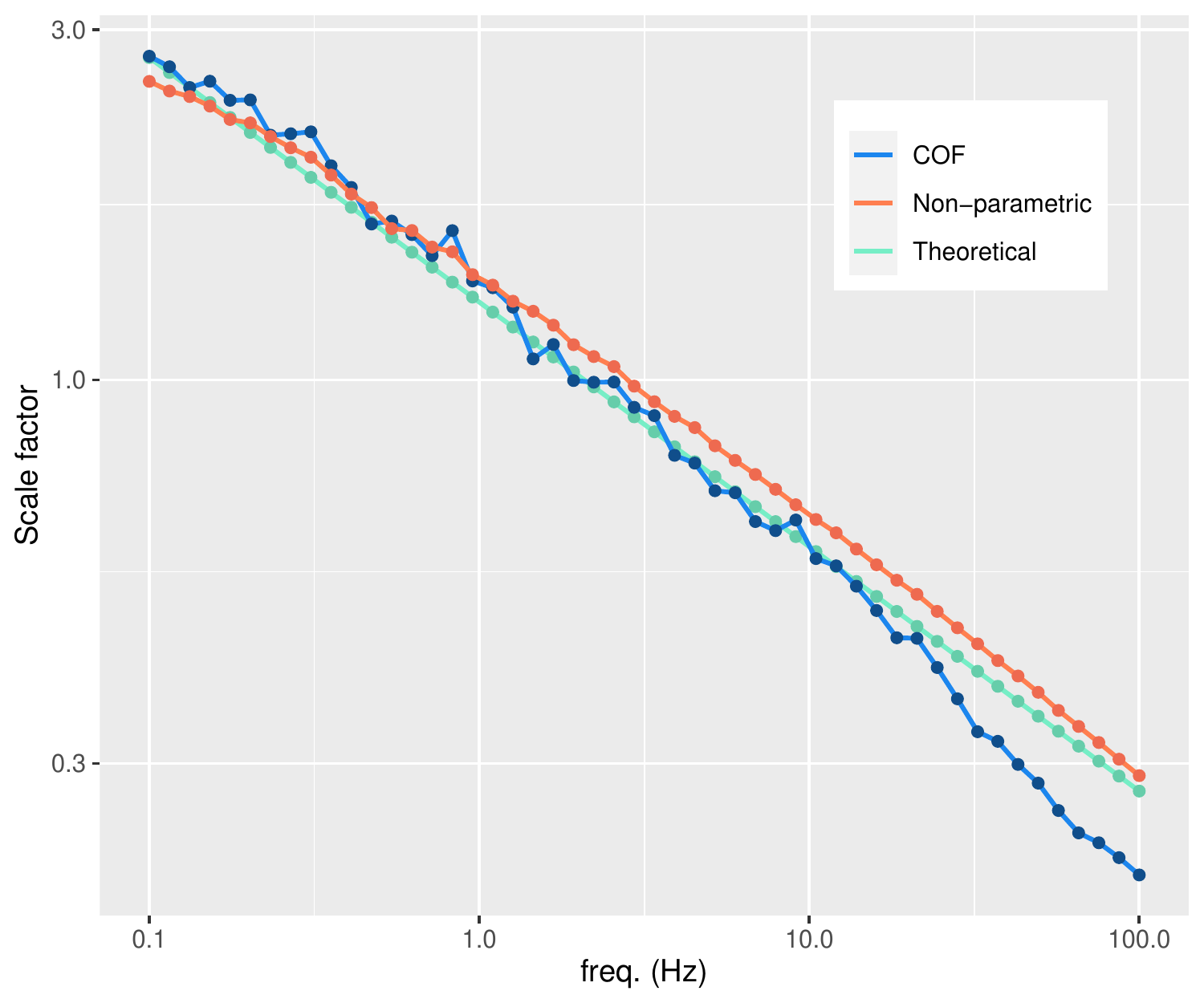}
\caption{}
\label{fig:tau_plot2}
\end{subfigure}
\caption{Brookhaven turbulence data: (a) Estimation of the smoothness parameter $\alpha$ using the change of frequency method, and (b) comparison of estimates of $\tau_n$ using change of frequency and non-parametric methods.} 
\end{figure}

\begin{figure}[t]
\begin{subfigure}{0.5\textwidth}
\includegraphics[scale = 0.5]{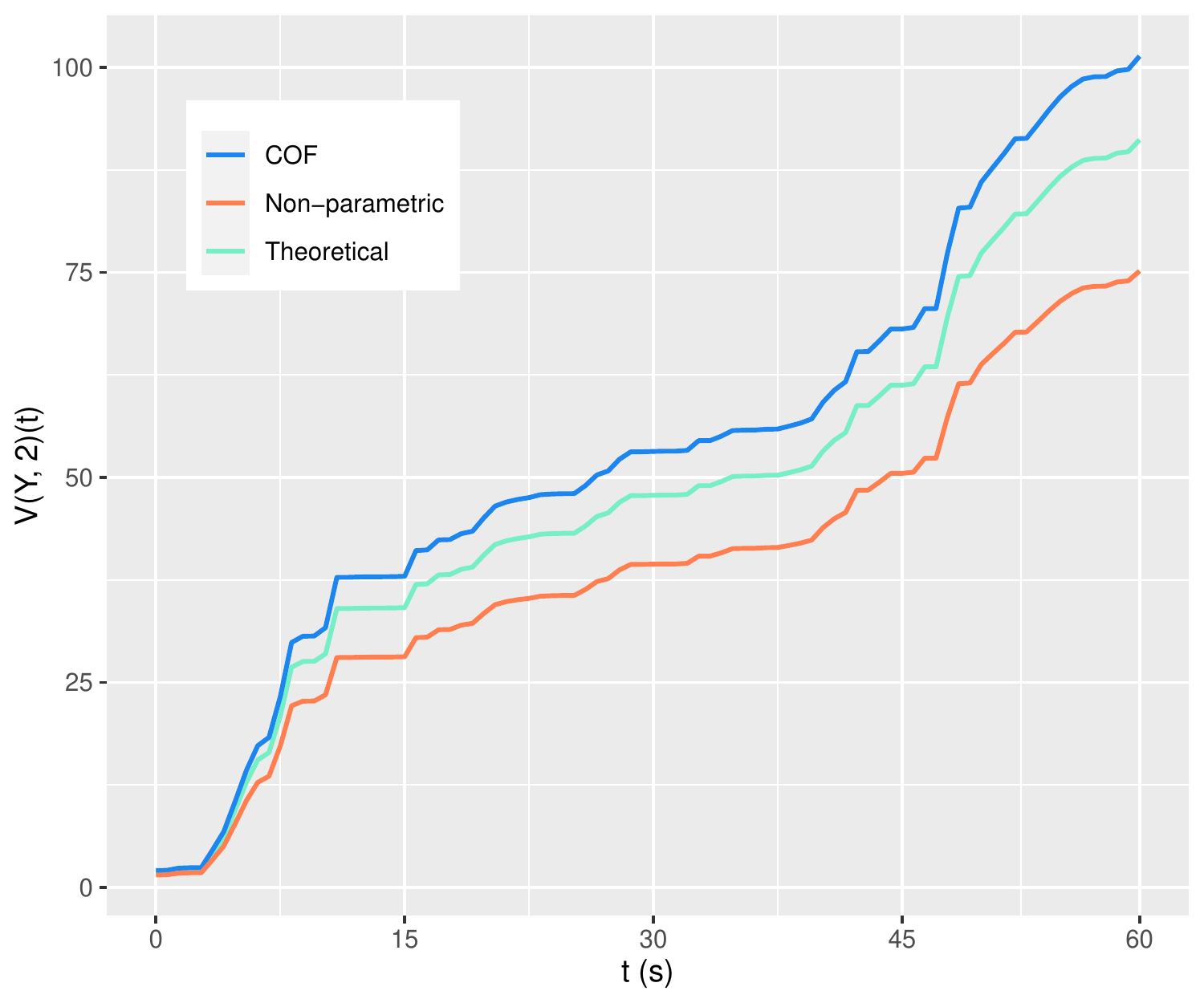}
\caption{}
\label{fig:int_vol}
\end{subfigure}
\begin{subfigure}{0.5\textwidth}
\includegraphics[scale = 0.5]{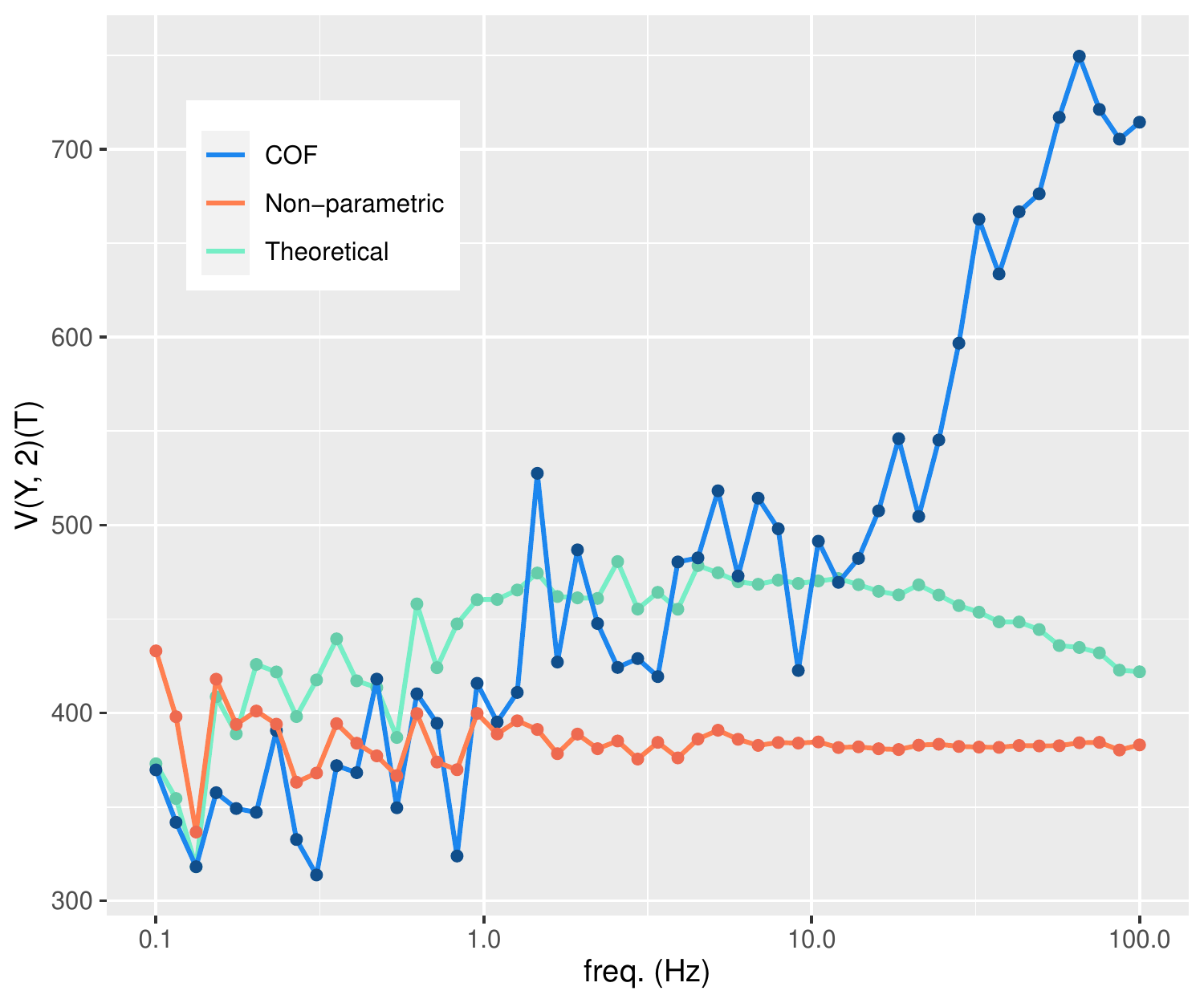}
\caption{}
\label{fig:acc_vol2}
\end{subfigure}
\caption{Brookhaven turbulence data: (a) Estimation of the integrated squared volatility over the first minute interval, using data sampled at $1.5Hz$, and (b) comparison of estimates of the integrated squared volatility at $T = 6$ minutes, averaged over five intervals of length 6 minutes.} 
\end{figure}

\section{Conclusion}
In this article we present and compare a number of estimators for the integrated power volatility process of a Brownian semistationary process. We have established a consistent non-parametric estimator for the scale factor used in the non-semi-martingale setting, which in turn gave us a feasible estimator for the integrated power volatility process, which also obeys a limit theorem. In the simulation study, we have compared the convergence properties of the three feasible estimators against the `ideal' benchmark of the infeasible estimator, where the scale factor is known. We have shown experimentally that the non-parametric estimator has the lowest error when estimating the integrated volatility, and converges well as $\Delta_n \rightarrow 0$.

We further demonstrated the convergence of the central limit theorem in the infeasible case, and tested whether the central limit theorem can be translated to the feasible case, by replacing all quantities with estimators, and showed that although the shape of the distribution appears Gaussian, the rate of convergence is slower than $\Delta_n^{1/2}$, so that a feasible confidence interval based on $z$ values scaled by $\Delta_n^{1/2}$ would not be appropriate. We finally apply the methods to estimation of the volatility in a turbulence dataset, and demonstrate that the nonparametric estimator produces results consistent with Kolmogorov's law across a range of frequencies, where the method produced by estimating the smoothness parameter may fail at higher frequencies due to overestimation of the smoothness of the paths.

\section{Proofs}
\subsection{Proof of Proposition \ref{prop:clt_scale_factor}}
\begin{proof}

We begin by squaring the estimator to obtain
\begin{align*}
\begin{split}
(\tau_n^N)^2 &= \frac{1}{N} \sum_{i=1}^{N} (\Delta_i^n X)^2 \\
&= \frac{1}{N} \sum_{i=1}^{N}  \left(X_{i\Delta_n} - X_{(i-1)\Delta_n}\right)^2 \\
&= \frac{1}{N} \sum_{i=1}^{N}  \big( X^2_{i\Delta_n} + X^2_{(i-1)\Delta_n} - 2X_{i\Delta_n} X_{(i-1)\Delta_n} \big) \\
&=  2 \hat{\gamma}(0) + \frac{1}{N}\left( X_0^2 - X_{N\Delta_n}^2\right) - 2\hat{\gamma}(\Delta_n).
\end{split}
\end{align*}

Therefore, writing $\mathbf{a} = (2, -2)^T$, $\boldsymbol{\gamma} = (\gamma(0), \gamma(\Delta_n))^T$ and $\hat{\boldsymbol{\gamma}} = (\hat{\gamma}(0), \hat{\gamma}(\Delta_n))^T$, we may write 

$$\sqrt{N} \left( (\tau_n^N)^2 - \tau_n^2 \right) = \sqrt{N}\mathbf{a}^T \left( \hat{\boldsymbol{\gamma}} - \boldsymbol{\gamma} \right) + \frac{1}{\sqrt{N}}\left( X_0^2 - X_{N\Delta_n}^2\right).$$

Now consider the term $\frac{1}{\sqrt{N}}\left( X_0^2 - X_{N\Delta_n}^2\right)$ first. By applying the triangle inequality and taking expectations we have 

$$\mathbb{E}\left[ \vert X_0^2 - X_{N\Delta_n}^2 \vert \right] \leq \mathbb{E}\left[X_0^2\right] +\mathbb{E}\left[ X_{N\Delta_n}^2 \right] = 2\gamma(0) < \infty $$
which implies that 

$$\frac{1}{\sqrt{N}} \mathbb{E} \left[ \vert  X_0^2 - X_{N\Delta_n}^2 \vert \right] \rightarrow 0 \qquad \text{as} \ N \rightarrow \infty$$
which shows the $L^1$ convergence of this term to zero. Then using the limit theorem for the sample autocovariance, we have

$$ \sqrt{N}\mathbf{a}^T \left( \hat{\boldsymbol{\gamma}} - \boldsymbol{\gamma} \right) \xrightarrow{d} N(0,\mathbf{a}^T V \mathbf{a}).$$

But $\mathbf{a}^T V \mathbf{a} = 4V_{00} - 8V_{01} + 4V_{11} := v$ and hence in the limit as $N \rightarrow \infty$, with an application of Slutsky's theorem, we have the convergence

$$ \sqrt{N} \left( (\tau_n^N)^2 - \tau_n^2 \right) \xrightarrow{d} N(0,v),$$
and then with an application of the Delta method with $f(\theta) = \theta^{1/2}$ and $\theta = \tau_n^2$, we have the claimed result.
\end{proof}

\subsection{Proof of Proposition \ref{prop:feasible_scale_factor_convergence}}

\begin{proof}
It is sufficient to prove that 

$$ \mathbb{E}\left[ \left| \left(\frac{\tau_n^{Y,N}}{\tau_n^Y}\right)^2 -1 \right|\wedge 1 \right] \leq n^{-(\frac{1}{2} + \delta)}.$$

Observe that
\begin{equation*}
\mathbb{E}\left[\left|\left(\frac{\tau_n^{Y,N}}{\tau_n^Y}\right)^{2}-1\right|\wedge 1 \right]\leq \frac{1}{(\tau_n^Y)^{2}}\mathbb{E}\left[\left| (\tau_n^{Y,N})^{2}-(\tau_n^Y)^{2}\right| \right]\leq \frac{1}{(\tau_n^Y)^{2}}\sqrt{\mathbb{E}\left[\left( (\tau_n^{Y,N})^{2}-(\tau_n^Y)^{2}\right)^{2} \right]}
\end{equation*}
\begin{equation*}
=\frac{1}{\mathbb{E}(\Delta_{1}^{n} Y)^{2}}\sqrt{\mathbb{E}\left[\left( \frac{1}{N}\sum_{i=1}^{N} (\Delta_{i}^{n} Y)^{2}-\mathbb{E}(\Delta_{1}^{n} Y)^{2}\right)^{2} \right]} 
\end{equation*}
\begin{equation*}
= \frac{1}{\mathbb{E}(\Delta_{1}^{n} Y)^{2}}\sqrt{\frac{1}{N^{2}}\sum_{i=1}^{N}\sum_{j=1}^{N}\mathbb{E}\left[\Big((\Delta^{n}_{i}Y)^{2}-\mathbb{E}\left[(\Delta^{n}_{1}Y)^{2}\right]\Big)\Big((\Delta^{n}_{j}Y)^{2}-\mathbb{E}\left[(\Delta^{n}_{1}Y)^{2}\right]\Big)\right]}
\end{equation*}
Let us concentrate now on the argument in the above square root. We show an explicit bound for
\begin{equation}\label{CCc}
\mathbb{E}\left[\Big((\Delta^{n}_{i}Y)^{2}-\mathbb{E}\left[(\Delta^{n}_{1}Y)^{2}\right]\Big)\Big((\Delta^{n}_{j}Y)^{2}-\mathbb{E}\left[(\Delta^{n}_{1}Y)^{2}\right]\Big)\right].
\end{equation}

Since the distribution of $\Delta^{n}_{1}Y$ conditioned on $\sigma$ is Gaussian, we may write $\Delta^{n}_{1}Y|\sigma\sim N(0,\Sigma)$ for some positive random variable $\Sigma$, and we have that for $i=j$
	\begin{equation*}
	\mathbb{E}\left[\left((\Delta^{n}_{i}Y)^{2}-\mathbb{E}\left[(\Delta^{n}_{1}Y)^{2}\right]\right)^{2}\right]=\mathbb{E}\left[(\Delta^{n}_{i}Y)^{4}\right]-\mathbb{E}\left[(\Delta^{n}_{1}Y)^{2}\right]^{2}
	\end{equation*}
		\begin{equation*}
=\mathbb{E}\left[\mathbb{E}\left[(\Delta^{n}_{1}Y)^{4}|\sigma\right]\right]-\mathbb{E}\left[\mathbb{E}\left[(\Delta^{n}_{1}Y)^{2}|\sigma\right]\right]^{2}=3\mathbb{E}[\Sigma^{2}]-\mathbb{E}[\Sigma]^{2}
		\end{equation*}
		\begin{equation}\label{0}
		\leq 3\mathbb{E}[\Sigma^{2}] \leq 4A_{1}A_{7}
		\end{equation}
		
		For $i>j$, the discussion is more articulated. First, observe that we have
	\begin{equation*}
	\mathbb{E}\left[\left((\Delta^{n}_{i}Y)^{2}-\mathbb{E}\left[(\Delta^{n}_{1}Y)^{2}\right]\right)\left((\Delta^{n}_{j}Y)^{2}-\mathbb{E}\left[(\Delta^{n}_{1}Y)^{2}\right]\right)\right]
	\end{equation*}	
	\begin{equation*}
	=\mathbb{E}\left[(\Delta^{n}_{i}Y)^{2}(\Delta^{n}_{j}Y)^{2}\right]-\mathbb{E}\left[(\Delta^{n}_{1}Y)^{2}\right]^{2}=\mathbb{E}\left[\mathbb{E}\left[(\Delta^{n}_{i}Y)^{2}(\Delta^{n}_{j}Y)^{2}|\sigma\right]\right]-\mathbb{E}\left[\mathbb{E}\left[(\Delta^{n}_{1}Y)^{2}|\sigma\right]\right]^{2}
	\end{equation*}
	then by Isserlis's theorem (also known as Wick's theorem) we have
	\begin{equation}\label{Wick}
	=\mathbb{E}\left[\mathbb{E}\left[(\Delta^{n}_{i}Y)^{2}|\sigma\right]\mathbb{E}\left[(\Delta^{n}_{j}Y)^{2}|\sigma\right]\right]+2\mathbb{E}\left[\mathbb{E}\left[(\Delta^{n}_{i}Y)(\Delta^{n}_{j}Y)|\sigma\right]^{2}\right]-\mathbb{E}\left[\mathbb{E}\left[(\Delta^{n}_{1}Y)^{2}|\sigma\right]\right]^{2}.
	\end{equation}
	Now, observe that 
	\begin{equation*}
	\mathbb{E}\left[\mathbb{E}\left[(\Delta^{n}_{i}Y)(\Delta^{n}_{j}Y)|\sigma\right]^{2}\right]=\mathbb{E}\Bigg[\Bigg(\int_{(j-1)\Delta_{n}}^{j\Delta_{n}} g(j\Delta_{n}-s)\left(g(\Delta_{n}i-s)-g(\Delta_{n}(i-1)-s)\right)\sigma_{s}^{2}ds
	\end{equation*}
	\begin{equation*}
	+\int_{-\infty}^{(j-1)\Delta_{n}}\left(g(j\Delta_{n}-s)-g((j-1)\Delta_{n}-s)\right)\left(g(\Delta_{n}i-s)-g(\Delta_{n}(i-1)-s)\right)\sigma_{s}^{2}ds\Bigg)^{2}\Bigg]
	\end{equation*}
	\begin{equation*}
	\leq A_{1} \Bigg(\int_{(j-1)\Delta_{n}}^{j\Delta_{n}} g(j\Delta_{n}-s)\left(g(\Delta_{n}i-s)-g(\Delta_{n}(i-1)-s)\right)ds
	\end{equation*}
	\begin{equation*}
	+\int_{-\infty}^{(j-1)\Delta_{n}}\left(g(j\Delta_{n}-s)-g((j-1)\Delta_{n}-s)\right)\left(g(\Delta_{n}i-s)-g(\Delta_{n}(i-1)-s)\right)ds\Bigg)^{2}
	\end{equation*}
	Then, by Cauchy-Schwarz inequality we have
	\begin{equation*}
	\leq A_{1} \Bigg(\sqrt{\int_{(j-1)\Delta_{n}}^{j\Delta_{n}} g(j\Delta_{n}-s)^{2}ds\int_{(j-1)\Delta_{n}}^{j\Delta_{n}}\left(g(\Delta_{n}i-s)-g(\Delta_{n}(i-1)-s)\right)^{2}ds}
	\end{equation*}
	\begin{equation*}
	+\sqrt{\int_{-\infty}^{(j-1)\Delta_{n}}g((j-1)\Delta_{n}-s)^{2}ds\int_{-\infty}^{(j-1)\Delta_{n}}\left(g(\Delta_{n}i-s)-g(\Delta_{n}(i-1)-s)\right)^{2}ds}\Bigg)^{2}
	\end{equation*}
	\begin{equation}\label{first}
	\leq 4A_{1} \int_{-\infty}^{0}g(-s)^{2}ds\int_{-\infty}^{j\Delta_{n}}g(\Delta_{n}(i-1)-s)^{2}ds
	\end{equation}
	Now, let us focus on the last integral. Recall that by assumption $g(x)< A_{5}e^{-A_{6}x}$ for $x> K$. Then, for $i-j<\frac{K}{\Delta_{n}}+1$ we have that
	\begin{equation*}
	\int_{-\infty}^{j\Delta_{n}}g(\Delta_{n}(i-1)-s)^{2}ds\leq \int_{-\infty}^{0}g(-s)^{2}ds\leq \int_{-\infty}^{0}g(-s)^{2}ds \frac{e^{-2\Delta_{n}(i-j)}}{e^{-2\Delta_{n}(\frac{K}{\Delta_{n}}+1)}}
	\end{equation*}
	while for $i-j\geq\frac{K}{\Delta_{n}}+1$ we have
	\begin{equation*}
	\int_{-\infty}^{j\Delta_{n}}g(\Delta_{n}(i-1)-s)^{2}ds=\int_{-\infty}^{j\Delta_{n}}e^{-2\Delta_{n}(i-1)+2s}ds=\frac{1}{2}e^{-2\Delta_{n}(i-j)}e^{2\Delta_{n}}
	\end{equation*}
	Hence, combining the above two cases we have 
	\begin{equation*}
	\int_{-\infty}^{j\Delta_{n}}g(\Delta_{n}(i-1)-s)^{2}ds\leq e^{-2\Delta_{n}(i-j)}e^{2\Delta_{n}}\left(1+e^{2K}\int_{-\infty}^{0}g(-s)^{2}ds \right)
	\end{equation*}
	Therefore, we have that $(\ref{first})$ is bounded by
	\begin{equation}\label{1}
e^{-2\Delta_{n}(i-j)}4A_{1}A_{7}e^{2\Delta_{n}}\left(1+e^{2K}A_{7} \right)
	\end{equation}
	
	Now let us look at the other terms in (\ref{Wick}). We have
	\begin{equation*}
	\mathbb{E}\left[\mathbb{E}\left[(\Delta^{n}_{i}Y)^{2}|\sigma\right]\mathbb{E}\left[(\Delta^{n}_{j}Y)^{2}|\sigma\right]\right]-\mathbb{E}\left[\mathbb{E}\left[(\Delta^{n}_{1}Y)^{2}|\sigma\right]\right]^{2}
	\end{equation*}
	\begin{equation*}
	=\mathbb{E}\Bigg[\Bigg(\int_{(i-1)\Delta_{n}}^{i\Delta_{n}} g(i\Delta_{n}-s)^{2} \sigma_{s}^{2} ds+\int_{-\infty}^{(i-1)\Delta_{n}} (g(i\Delta_{n}-s)-g((i-1)\Delta_{n}-s))^{2} \sigma^{2}_{s} ds\Bigg)
	\end{equation*}
	\begin{equation*}
	\Bigg(\int_{(j-1)\Delta_{n}}^{j\Delta_{n}} g(j\Delta_{n}-s)^{2} \sigma_{s}^{2} ds+\int_{-\infty}^{(j-1)\Delta_{n}} (g(j\Delta_{n}-s)-g((j-1)\Delta_{n}-s))^{2} \sigma^{2}_{s} ds\Bigg)\Bigg]-\mathbb{E}[\Sigma]^{2},
	\end{equation*}
	\begin{equation*}
	=\mathbb{E}\Bigg[\Bigg(\int_{0}^{\Delta_{n}} g(\Delta_{n}-s)^{2} \sigma_{s+(i-1)\Delta_{n}}^{2} ds+\int_{-\infty}^{0} (g(\Delta_{n}-s)-g(-s))^{2} \sigma^{2}_{s+(i-1)\Delta_{n}} ds\Bigg)
	\end{equation*}
	\begin{equation*}
	\Bigg(\int_{0}^{\Delta_{n}} g(\Delta_{n}-s)^{2} \sigma_{s+(j-1)\Delta_{n}}^{2} ds+\int_{-\infty}^{0} (g(\Delta_{n}-s)-g(-s))^{2} \sigma^{2}_{s+(j-1)\Delta_{n}} ds\Bigg)\Bigg]
	\end{equation*}
	\begin{equation*}
	-\mathbb{E}\Bigg[\int_{0}^{\Delta_{n}} g(\Delta_{n}-s)^{2} \sigma_{s}^{2} ds+\int_{-\infty}^{0} (g(\Delta_{n}-s)-g(-s))^{2} \sigma^{2}_{s} ds\Bigg]^{2}
	\end{equation*}
	\begin{equation*}
	=\int_{0}^{\Delta_{n}}\int_{0}^{\Delta_{n}} g(\Delta_{n}-s)^{2}g(\Delta_{n}-r)^{2} \left(\mathbb{E}[\sigma_{s+(i-1)\Delta_{n}}^{2} \sigma_{r+(j-1)\Delta_{n}}^{2}]-\mathbb{E}[\sigma^{2}_{1}]^{2}\right)drds
	\end{equation*}
	\begin{equation*}
	+\int_{-\infty}^{0}\int_{-\infty}^{0} (g(\Delta_{n}-s)-g(-s))^{2}(g(\Delta_{n}-r)-g(-r))^{2} \left(\mathbb{E}[\sigma_{s+(i-1)\Delta_{n}}^{2} \sigma_{r+(j-1)\Delta_{n}}^{2}]-\mathbb{E}[\sigma^{2}_{1}]^{2}\right)drds
	\end{equation*}
	\begin{equation*}
	+\int_{0}^{\Delta_{n}}\int_{-\infty}^{0}g(\Delta_{n}-s)^{2}(g(\Delta_{n}-r)-g(-r))^{2}\left(\mathbb{E}[\sigma_{s+(i-1)\Delta_{n}}^{2} \sigma_{r+(j-1)\Delta_{n}}^{2}]-\mathbb{E}[\sigma^{2}_{1}]^{2}\right)drds
	\end{equation*}
	\begin{equation}\label{integrability}
	+\int_{0}^{\Delta_{n}}\int_{-\infty}^{0}g(\Delta_{n}-r)^{2}(g(\Delta_{n}-s)-g(-s))^{2}\left(\mathbb{E}[\sigma_{s+(i-1)\Delta_{n}}^{2} \sigma_{r+(j-1)\Delta_{n}}^{2}]-\mathbb{E}[\sigma^{2}_{1}]^{2}\right)drds
	\end{equation}
	Now by assumption we have that
	\begin{equation*}
	\mathbb{E}[\sigma_{s+(i-1)\Delta_{n}}^{2} \sigma_{r+(j-1)\Delta_{n}}^{2}]-\mathbb{E}[\sigma^{2}_{1}]^{2}\leq A_{3}\exp\left(-A_{4}|s+(i-1)\Delta_{n}-r+(j-1)\Delta_{n}|\right)
	\end{equation*}
	\begin{equation*}
	\leq A_{3}\exp(A_{4}(-\Delta_{n}(i-j)-s+r))=A_{3} e^{-A_{4}\Delta_{n}(i-j)} e^{-A_{4}s+A_{4}r}
	\end{equation*}
	Moreover, observe that for each $s\in(-\infty,-K)$ we have
	\begin{equation*}
	g(-s)^{2}e^{-s}\leq A_{5}e^{(2A_{6}-A_{4})s}
	\end{equation*}
	but since $2A_{6}>A_{4}$ by assumption we have that all the integrals in $(\ref{integrability})$ are finite. In particular, $(\ref{integrability})$ is bounded by
	\begin{equation*}
	4A_{3}e^{-A_{4}\Delta_{n}(i-j)}\int_{-\infty}^{\Delta_{n}}\int_{-\infty}^{\Delta_{n}} g(\Delta_{n}-s)^{2}g(\Delta_{n}-r)^{2}e^{-A_{4}s+A_{4}r} drds
	\end{equation*}
	\begin{equation*}
	\leq 4e^{-A_{4}\Delta_{n}(i-j)}A_{3}A_{7}\int_{-\infty}^{\Delta_{n}} g(\Delta_{n}-s)^{2}e^{-A_{4}s} ds
	\end{equation*}
	\begin{equation*}
	\leq 4e^{-A_{4}\Delta_{n}(i-j)}A_{3}A_{7}\left(\int_{0}^{\Delta_{n}} g(\Delta_{n}-s)^{2}e^{-A_{4}s} ds+\int_{-\infty}^{0} g(-s)^{2}e^{-A_{4}s} ds\right)
	\end{equation*}
	\begin{equation}\label{2}
	\leq e^{-A_{4}\Delta_{n}(i-j)}\frac{8A_{7}A_{5}A_{3}}{2A_{6}-A_{4}}.
	\end{equation}
	Thus, using $(\ref{0})$, $(\ref{1})$ and $(\ref{2})$ we obtain that $(\ref{CCc})$ is bounded by
	\begin{equation*}
	e^{-2\Delta_{n}(i-j)}4A_{1}A_{7}e^{2\Delta_{n}}\left(1+e^{2K}A_{7} \right) +e^{-A_{4}\Delta_{n}(i-j)}\frac{8A_{7}A_{5}A_{3}}{2A_{6}-A_{4}}
	\end{equation*}
		\begin{equation*}
		\leq e^{-(2\wedge A_{4})\Delta_{n}(i-j)}4A_{7}\left(A_{1}e^{2\Delta_{n}}\left(1+e^{2K}A_{7} \right) +\frac{2A_{5}A_{3}}{2A_{6}-A_{4}}\right)=C_{n}e^{-D_{n}(i-j)}
		\end{equation*}
	where we have used that $(\ref{0})$ is smaller than $(\ref{1})$ (so that we cover the case of $i=j$ by using the bounds for $i>j$). Then, we have that
	\begin{equation*}
	\mathbb{E}\left[\left|\left(\frac{\tau_n^{Y,N}}{\tau_n^Y}\right)^{2}-1\right|\wedge 1 \right]\leq \frac{1}{\tilde{A}_{n}}\sqrt{\frac{1}{N^{2}}\sum_{i=1}^{N}\sum_{j=1}^{N}C_{n}e^{-D_{n}(i-j)}}
	\end{equation*}
	\begin{equation*}
	=\frac{\sqrt{C_{n}}}{N\tilde{A}_{n}}\sqrt{N+\frac{2e^{-D_{n}}}{(e^{-D_{n}}-1)^{2}}(e^{-(N+1)D_{n}}+1)+\frac{2e^{-D_{n}}}{e^{-D_{n}}-1}N}
	\end{equation*}
	\begin{equation*}
	\leq \frac{1}{\sqrt{N}}\frac{\sqrt{C_{n}}}{\tilde{A}_{n}}\sqrt{1+\frac{2e^{-D_{n}}}{(e^{-D_{n}}-1)^{2}}(e^{-D_{n}}+1)}\leq \frac{1}{\sqrt{N}}\frac{3\sqrt{C_{n}}}{\tilde{A}_{n}(1-e^{-D_{n}})} =\frac{1}{\sqrt{N}}E_{n}
	\end{equation*}
	and so it converges to zero as $N\to\infty$. Notice that $E_{n}<C\frac{e^{\Delta_{n}}}{\tilde{A}_{n}(1-e^{-D_{n}})}<\frac{C'}{\tilde{A}_{n}(1-e^{-D_{n}})}$, where $C$ and $C'$ do not depend on $n$.
	
	Finally, we have that 
	\begin{equation*}
	\frac{1}{\sqrt{N}}E_{n}\leq n^{-(\frac{1}{2}+\delta)}\Rightarrow N \geq E_{n}^{2}n^{1+2\delta}.
	\end{equation*}
\end{proof}

\subsection{Proof of Theorem \ref{thm:feasible_clt}}

\begin{proof}
By Proposition \ref{prop:feasible_scale_factor_convergence} we have that
\begin{equation}\label{rhoDelta}
\mathbb{E}\left[\left|\Delta_n^{-1/2}\left[\frac{(\tau_n^{Y,N})^{-p}}{(\tau_n^{Y})^{-p}}-1\right]\right|\wedge 1\right]<n^{-\delta}.
\end{equation}
Thus, we have that
\begin{equation*}
\Delta_n^{-1/2}\left[\frac{(\tau_n^{Y,N})^{-p}}{(\tau_n^{Y})^{-p}}-1\right]\stackrel{\mathbb{P}}{\rightarrow}0,\quad\textnormal{as $n\to\infty$}.
\end{equation*}
From Theorem \ref{thm:infeasible_weak_law} we know that $\Delta_n (\tau_n^{Y})^{-p}V(Y,p;\Delta_n)(t)$ converges in probability and so by Slutsky's theorem we have
\begin{equation*}
\Delta_n^{-1/2}\left[\frac{(\tau_n^{Y,N})^{-p}}{(\tau_n^{Y})^{-p}}-1\right]\Delta_n (\tau_n^{Y})^{-p}V(Y,p;\Delta_n)(t)\stackrel{\mathbb{P}}{\rightarrow}0,
\end{equation*}
and hence,
\begin{equation*}
\Delta_n^{1/2}(\tau_n^{Y,N})^{-p}V(Y,p;\Delta_n)(t)-\Delta_n^{1/2}(\tau_n^Y)^{-p}V(Y,p;\Delta_n)(t)\stackrel{\mathbb{P}}{\rightarrow}0.
\end{equation*}
Finally, from Theorem \ref{thm:infeasible_clt} we obtain the stated result. In particular, we have that
\begin{equation*}
	\Delta_n^{-1/2}\left(\Delta_n (\tau_n^{Y,N})^{-p}V(Y,p;\Delta_n)(t) -[\mathbb{E}(\sigma_1^2)]^{-p/2}V(Y,p)(t)\right)
\end{equation*}
\begin{equation*}
=\Delta_n^{-1/2}\left(\Delta_n (\tau_n^Y)^{-p}V(Y,p;\Delta_n)(t) -[\mathbb{E}(\sigma_1^2)]^{-p/2}V(Y,p)(t)\right)
\end{equation*}
\begin{equation*}
+\Delta_n^{1/2}(\tau_n^{Y,N})^{-p}V(Y,p;\Delta_n)(t)-\Delta_n^{1/2}(\tau_n^Y)^{-p}V(Y,p;\Delta_n)(t)
\end{equation*}
\begin{equation*}
=[\mathbb{E}(\sigma_1^2)]^{-p/2}\Delta_n^{-1/2}\left(\Delta_n \tau_n^{-p}V(Y,p;\Delta_n)(t) -V(Y,p)(t)\right)
\end{equation*}
\begin{equation}\label{final}
+\Delta_n^{1/2}(\tau_n^{Y,N})^{-p}V(Y,p;\Delta_n)(t)-\Delta_n^{1/2}(\tau_n^Y)^{-p}V(Y,p;\Delta_n)(t).
\end{equation}
The first summand in (\ref{final}) converges stably as showed in Theorem \ref{thm:infeasible_clt}, while the difference of the other two summands go to zero in probability as shown in this proof. Thus, we obtain the statement.
\end{proof}

\bibliographystyle{abbrv}
\bibliography{bibs/bibliography}

\end{document}